\documentclass[11pt]{amsart}

\usepackage{amssymb}
\usepackage{a4wide}
\usepackage{bbm}
\usepackage{graphicx}
\usepackage{paralist}

\usepackage{xcolor}
\definecolor{darkblue}{rgb}{0,0,0.7}
\newcommand{\darkblue}{\color{darkblue}}
\usepackage[pdftex]{hyperref}
\hypersetup{colorlinks=true, citecolor=darkblue, linkcolor=darkblue}

\newtheorem{theorem}{Theorem}[section]
\newtheorem{proposition}[theorem]{Proposition}
\newtheorem{lemma}[theorem]{Lemma}
\newtheorem{corollary}[theorem]{Corollary}

\newtheorem{definition}[theorem]{Definition}
\theoremstyle{definition}
\newtheorem{example}[theorem]{Example}
\newtheorem{remark}[theorem]{Remark}

\newcommand{\dotprod}[2]{\ensuremath{#1\cdot #2}}
\newcommand{\sub}[1]{\ensuremath{\underline{#1}}}
\newcommand{\conv}{\ensuremath{\mathrm{conv}}}
\newcommand{\KG}{\ensuremath{\mathrm{KG}}}
\newcommand{\simplex}{\ensuremath{\triangle}}
\newcommand{\R}{\ensuremath{\mathbb{R}}}
\newcommand{\N}{\ensuremath{\mathbb{N}}}
\newcommand{\Lfloor}{\left\lfloor}
\newcommand{\Rfloor}{\right\rfloor}
\newcommand{\fracfloor}[2]{\Lfloor\frac{#1}{#2}\Rfloor}
\newcommand{\bn}{{\overline{n}}}
\newcommand{\bm}{{\overline{m}}}
\newcommand{\bchi}{{\overline{\chi}}}
\newcommand{\polar}{\triangle}
\newcommand{\cK}{\mathcal{K}}
\newcommand{\cF}{\mathcal{F}}
\newcommand{\cZ}{\mathcal{Z}}
\newcommand{\defA}{A^{\sim}}
\newcommand{\defb}{b^{\sim}}
\newcommand{\defP}{P^{\sim}}
\newcommand{\ie}{{\it i.e.},~}
\newcommand{\eqdef}{\mbox{~\raisebox{0.2ex}{\scriptsize\ensuremath{\mathrm:}}\ensuremath{=} }}
\newcommand{\eqfed}{\mbox{~\ensuremath{=}\raisebox{0.15ex}{\scriptsize\ensuremath{\mathrm:}} }}
\newcommand{\defn}[1]{\emph{\darkblue #1}}

\graphicspath{{figures/}}

\title{Prodsimplicial-Neighborly Polytopes}

\author[B. Matschke]{Benjamin Matschke}
\address{Technische Universit\"at Berlin, Germany}
\email{benjaminmatschke@googlemail.com}

\author[J. Pfeifle]{Julian Pfeifle}
\address{Departament de Matem\`atica Aplicada II, Universitat Polit\`ecnica de Catalunya, Barcelona, Spain}
\email{julian.pfeifle@upc.edu}

\author[V. Pilaud]{Vincent Pilaud}
\address{\'Equipe Combinatoire et Optimisation, Universit\'e Pierre et Marie Curie, Paris, France}
\email{vpilaud@math.jussieu.fr}

\thanks{Benjamin Matschke was supported by DFG research group
Polyhedral Surfaces and by Deutsche Telekom Stiftung. Julian Pfeifle was supported by grants MTM2006-01267 and MTM2008-03020 from the Spanish Ministry of Education and Science 
and 2009SGR1040 from the Generalitat de Catalunya. Vincent Pilaud was partially supported by grant MTM2008-04699-C03-02 of the Spanish Ministry of Education and Science.}

\begin{document}

\begin{abstract}
  Simultaneously generalizing both neighborly and neighborly cubical
  polytopes, we introduce \defn{PSN polytopes}: their $k$-skeleton is
  combinatorially equivalent to that of a product of $r$~simplices.

  We construct PSN~polytopes by three different methods, the most
  versatile of which is an extension of Sanyal \& Ziegler's
  ``projecting deformed products'' construction to products of
  arbitrary simple polytopes. For general $r$ and $k$, the lowest
  dimension we achieve is $2k+r+1$.

  Using topological obstructions similar to those introduced by Sanyal
  to bound the number of vertices of Minkowski sums, we show that this
  dimension is minimal if we additionally require that the
  PSN~polytope is obtained as a projection of a polytope that is
  combinatorially equivalent to the product of $r$ simplices, when the
  dimensions of these simplices are all large \mbox{compared to~$k$.}
\end{abstract}

\maketitle

\section{Introduction}

\subsection{Definitions}

Let $\simplex_n$ denote the $n$-dimensional simplex. For any tuple
$\sub{n} \eqdef (n_1,\dots,n_r)$ of integers, we denote by
$\simplex_{\sub{n}}$ the product of simplices
$\simplex_{n_1}\times\dots\times\simplex_{n_r}$. This is a polytope
of dimension $\sum n_i$, whose non-empty faces are obtained as
products of non-empty faces of the simplices
$\simplex_{n_1},\dots,\simplex_{n_r}$. For example,
Figure~\ref{fig:prodSimplGraphs} represents the graphs
of~$\simplex_i\times\simplex_6$, for~${i\in \{1,2,3\}}$.

\begin{figure}[htbp]
   \centerline{\includegraphics[scale=.7]{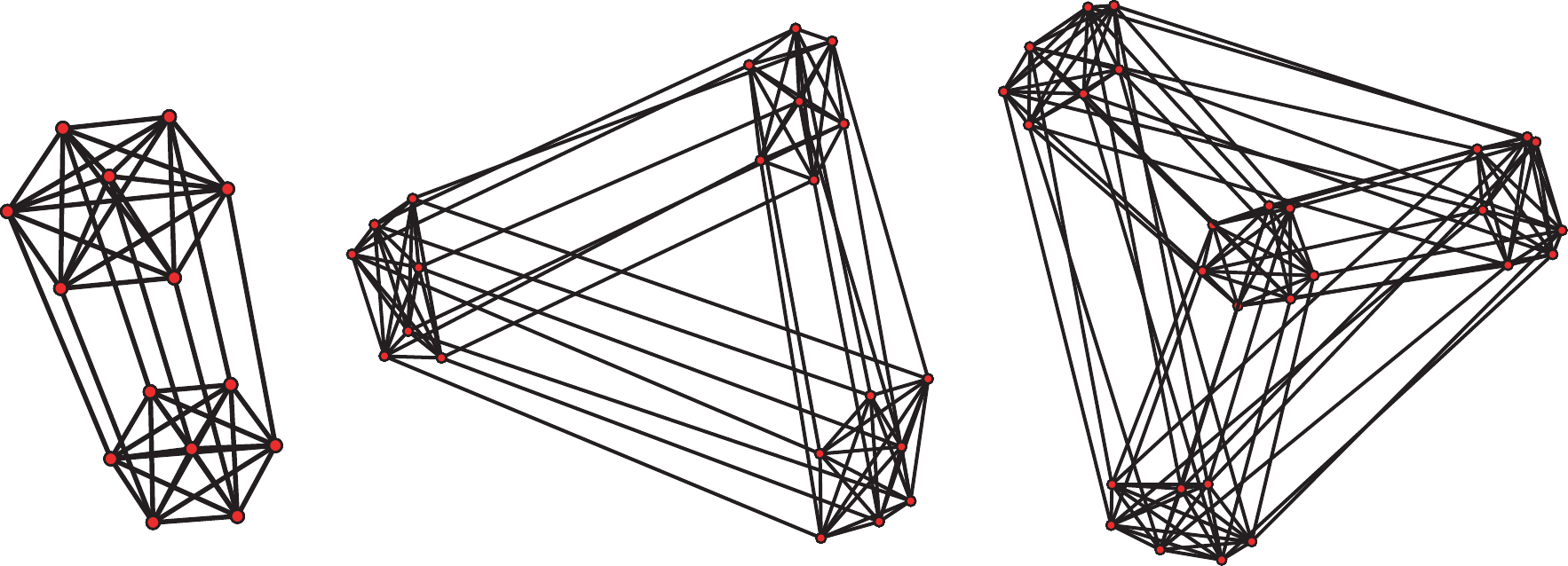}}
   \caption{{The graphs of the products
   $\simplex_{(i,6)}=\simplex_i\times\simplex_6$, for
   $i\in \{1,2,3\}$.}}
   \label{fig:prodSimplGraphs}
\end{figure}

We are interested in polytopes with the same ``initial'' structure as
these products.

\begin{definition}
  Let $k\ge0$ and $\sub{n} \eqdef (n_1,\dots,n_r)$, with $r\ge1$ and $n_i\ge
  1$ for all~$i$. A~convex polytope in some Euclidean space is
  \defn{$(k,\sub{n})$-prodsimplicial-neighborly}~---~or
  \defn{$(k,\sub{n})$-PSN} for short~---~if
  its $k$-skeleton is combinatorially equivalent to that of
  $\simplex_{\sub{n}} \eqdef \simplex_{n_1}\times\dots\times\simplex_{n_r}$.
\end{definition}

We choose the term ``prodsimplicial'' to shorten ``product
simplicial''.  This definition is essentially motivated by two
particular classes of PSN polytopes:
\begin{enumerate}
\item \defn{neighborly} polytopes arise when $r=1$;
\item \defn{neighborly cubical} polytopes~\cite{js-ncps-05,sz-capdd} arise when
  $\sub{n}=(1,1,\dots,1)$.
\end{enumerate}

\begin{remark}
  In the literature, a polytope is \defn{$k$-neighborly} if any subset
  of at most $k$ of its vertices forms a face. Observe that such a
  polytope is $(k-1,n)$-PSN with our notation.
\end{remark}

The product $\simplex_{\sub{n}}$ is a $(k,\sub{n})$-PSN
polytope of dimension $\sum n_i$, for each~$k$ with ${0\le k\le \sum
n_i}$. We are naturally interested in finding $(k,\sub{n})$-PSN
polytopes of smaller dimensions.  For example, the cyclic polytope
$C_{2k+2}(n+1)$ is a $(k,n)$-PSN~polytope of dimension~$2k+2$.  We
denote by~$\delta(k,\sub{n})$ the smallest possible dimension of a
$(k,\sub{n})$-PSN polytope.

\medskip
PSN polytopes can be obtained by projecting the
product~$\simplex_{\sub{n}}$, or a combinatorially equivalent
polytope, onto a smaller subspace. For example, the cyclic polytope
$C_{2k+2}(n+1)$ (just like any polytope with $n+1$ vertices) can be
seen as a projection of the simplex $\simplex_n$ to $\R^{2k+2}$.

\begin{definition}
  A $(k,\sub{n})$-PSN polytope is
  \defn{$(k,\sub{n})$-projected-prodsimplicial-neighborly}~---~or
  \defn{$(k,\sub{n})$-PPSN} for short~---~if it is a projection of a
  polytope that is combinatorially equivalent to $\simplex_{\sub{n}}$.
\end{definition}

We denote by $\delta_{pr}(k,\sub{n})$ the smallest possible dimension of
a $(k,\sub{n})$-PPSN polytope.

\subsection{Outline and main results}

The present paper may be naturally divided into two parts. In the
first part, we present three methods for constructing low-dimensional
PPSN polytopes:
\begin{enumerate}
\item Reflections of cyclic polytopes;
\item Minkowski sums of cyclic polytopes;
\item Deformed Product constructions in the spirit of Sanyal \&
  Ziegler~\cite{z-ppp-04,sz-capdd}.
\end{enumerate}
The second part derives topological obstructions for the existence of
such objects, using techniques developed by Sanyal
in~\cite{s-tovnms-09} (see also \cite{rs-npps}) to bound the number of
vertices of Minkowski sums.  In view of these obstructions, our
constructions in the first part turn out to be optimal for a wide
range of parameters.

\bigskip 
We devote the remainder of the introduction to highlighting
our most relevant results. To facilitate the
navigation in the article, we label each result by the number
it actually receives later on.

\bigskip
\paragraph{\bf Constructions.}

Our first non-trivial example is a $(k,(1,n))$-PSN polytope in
dimension ${2k+2}$, obtained by reflecting the cyclic polytope
$C_{2k+2}(n+1)$ through a well-chosen~hyperplane:

\medskip\noindent\textbf{Proposition \ref{prop:reflect}.} {\it
  For any $k\ge0$, $n\ge2k+2$ and $\lambda\in\R$ sufficiently large, the polytope
  \[
     P
     \ \eqdef \
     \conv\left(\big\{(t_i,\dots,t_i^{2k+2})^T\ | \ 
     i\in[n+1]\big\}
     \ \cup \
     \big\{(t_i,\dots,t_i^{2k+1},\lambda-t_i^{2k+2})^T\ | \ 
     i\in[n+1]\big\}\right)
  \]
  is a $(k,(1,n))$-PSN polytope of dimension $2k+2$.
}

\medskip
For example, this provides us with
a $4$-dimensional polytope whose graph is the cartesian product
$K_2\times K_n$, for any $n\ge3$.

\medskip
Next, forming a well-chosen Minkowski sum of cyclic polytopes yields
explicit coordinates for $(k,\sub{n})$-PPSN polytopes:

\medskip\noindent\textbf{Theorem \ref{theo:UBminkowskiCyclic}.} {\it
  Let $k\ge0$ and $\sub{n} \eqdef (n_1,\dots,n_r)$ with $r\ge1$ and $n_i\ge1$ for
  all $i$.  There exist index sets $I_1,\dots, I_r\subset\R$, with
  $|I_i|=n_i$ for all~$i$, such that the polytope
  \[
     P
     \ \eqdef \ 
     \conv\{w_{a_1,\dots,a_r}\  |\ (a_1,\dots,a_r)\in I_1\times\dots\times
     I_r\} 
     \ \subset \ 
     \R^{2k+r+1}
  \]
  is $(k,\sub{n})$-PPSN, where
  $
    w_{a_1,\dots,a_r}
    \eqdef
    \big(a_1,\dots,a_r,\sum_{i\in[r]}
      a_i^2,\dots,\sum_{i\in[r]} a_i^{2k+2}\big)^T.
  $
  Consequently,
  \[
     \delta(k,\sub{n})
     \ \le \ 
     \delta_{pr}(k,\sub{n})
     \ \le \
     2k+r+1.
  \]
}

For $r=1$ we recover neighborly polytopes. 

\medskip Finally, we extend Sanyal \& Ziegler's technique of
``projecting deformed products of polygons''~\cite{z-ppp-04,sz-capdd}
to products of arbitrary simple polytopes: given a polytope~$P$ that
is combinatorially equivalent to a product of simple polytopes, we
exhibit a suitable projection that preserves the complete $k$-skeleton
of~$P$.  More concretely, we describe how to use colorings of the
graphs of the polar polytopes of the factors in the product to raise
the dimension of the preserved skeleton. The basic version of this
technique yields the following result:

\medskip
\noindent\textbf{Proposition \ref{prop:defp2}.} {\it
  Let $P_1,\dots,P_r$ be simple polytopes of respective
  dimension~$n_i$, and with $m_i$~many
  facets. Let~$\chi_i \eqdef \chi(\text{sk}_1 P_i^\polar)$ denote the
  chromatic number of the graph of the polar polytope
  $P_i^\polar$. For a fixed integer $d\le\sum_{i=1}^r n_i$, let
  $t$ be maximal such that $\sum_{i=1}^t n_i\le d$. Then there exists
  a $d$-dimensional polytope whose $k$-skeleton is combinatorially
  equivalent to that of the product $P_1\times\dots\times P_r$
  provided}
\[
  0 \ \le \ k \ \le \ \sum_{i=1}^r (n_i-m_i) + \sum_{i=1}^t
  (m_i-\chi_i)+
  \Lfloor\frac12\left(d-1+\sum_{i=1}^t(\chi_i-n_i)\right)\Rfloor.
\]
\medskip  

A family of polytopes that minimize the last summand are products of
\defn{even polytopes} (all 2-dimensional faces have an even
number of vertices). See Example~\ref{ex:even} for the details, and
the end of Section~\ref{sec:general} for extensions of this technique.

\medskip
Specializing the factors to simplices provides another
construction of PPSN~polytopes. When some of these simplices are small
compared to~$k$, this technique in fact yields our best examples of
PPSN~polytopes:

\medskip\noindent\textbf{Theorem \ref{theo:defp-ppsn}.} \begingroup\itshape
For any $k\ge0$ and $\sub{n} \eqdef (n_1,\dots,n_r)$ with ${1=n_1=\dots=n_s<n_{s+1}\le\dots\le n_r}$,
\[
   \delta_{pr}(k,\sub{n})
   \ \le \
   \begin{cases}
     2(k+r)-s-t   & \text{if } 3s \le 2k+2r, \\
     2(k+r-s)+1   & \text{if } 3s = 2k+2r+1, \\
     2(k+r-s+1)   & \text{if } 3s \ge 2k+2r+2,
   \end{cases}
\]
where $t\in\{s,\dots,r\}$ is maximal such that  $3s+\sum_{i=s+1}^{t}(n_i+1)\ \le \ 2k+2r$.
\endgroup
\medskip

If $n_i=1$ for all $i$, we recover the neighborly cubical polytopes of
\cite{sz-capdd}.

\medskip
\paragraph{\bf Obstructions}

In order to derive lower bounds on the minimal dimension
$\delta_{pr}(k,\sub{n})$ that a \mbox{$(k,\sub{n})$-PPSN} polytope can
have, we apply and extend a method due to
Sanyal~\cite{s-tovnms-09}. For any projection which preserves the
$k$-skeleton of $\simplex_{\sub{n}}$, we use Gale duality to construct
a simplicial complex that can be embedded in a certain dimension. The
argument is then a topological obstruction based on Sarkaria's
criterion for the embeddability of a simplicial complex in terms of
colorings of Kneser graphs \cite{m-ubut-03}. We obtain the following
result:

\medskip\noindent\textbf{Corollary \ref{cor:topObstr}.} \begingroup\itshape
Let $\sub{n} \eqdef (n_1,\dots,n_r)$ with $1=n_1=\dots=n_s<n_{s+1}\le\dots\le
n_r$. 
  \begin{enumerate}
  \item If 
  \[
   0\ \le \ k \ \le \ \sum_{i=s+1}^r \fracfloor{n_i-2}{2} + 
   \max\left\{0,\fracfloor{s-1}{2}\right\}, 
  \]
  then $\delta_{pr}(k,\sub{n}) \ge 2k+r-s+1$.

  \item If $k\ge \Lfloor\frac12 \sum_i n_i\Rfloor$ then $\delta_{pr}(k,\sub{n})\ge \sum_i n_i$.
  \end{enumerate}
\endgroup

In particular, the upper and lower bounds provided by
Theorem~\ref{theo:UBminkowskiCyclic} and Corollary~\ref{cor:topObstr}
match over a wide range of parameters:

\begin{theorem}
\label{theo:mainResult}
Let $\sub{n} \eqdef (n_1,\dots,n_r)$ with $r\ge1$ and $n_i\ge2$ for
all~$i$. For any~$k$ such that $0\le k\le \sum_{i\in [r]}
\fracfloor{n_i-2}{2}$, the smallest $(k,\sub{n})$-PPSN polytope has
dimension exactly~$2k+r+1$. In other words:
  \[ 
    \delta_{pr}(k,\sub{n}) 
    \ = \
    2k+r+1.
  \]
\end{theorem}

\begin{remark}
  During the final stages of completing this paper, we learned that
  R\"orig and Sanyal~\cite{rs-npps} also applied Sanyal's topological
  obstruction method to derive lower bounds on the target dimension of
  a projection preserving skeleta of different kind of products
  (products of polygons, products of simplices, and wedge products of
  polytopes). In particular, for a product
  $\simplex_n\times\dots\times\simplex_n$ of $r$ identical simplices, $r\geq 2$,
  they obtain our Theorem~\ref{theo:topObstr-small-k} and a result
  (their Theorem~4.5) that is only slightly weaker than 
  Theorem~\ref{theo:topObstr-all-k} in this setting (compare with Sections 
  \ref{sec:ExplicitLowerBounds} and \ref{sec:ComparisonWithRSBounds}).
\end{remark}

\section{Constructions from cyclic polytopes}

Let $t\mapsto \mu_d(t) \eqdef (t,t^2,\dots,t^d)^T$ be the \defn{moment
  curve} in $\mathbb{R}^d$, $t_1<t_2<\dots<t_n$ be $n$ distinct real
numbers and $C_d(n) \eqdef \conv\{\mu_d(t_i)\ | \ i\in[n]\}$ denote the
\defn{cyclic polytope} in its realization on the moment curve. We
refer to~\cite[Theorem 0.7]{z-lp-95}
and~\cite[Corollary~6.1.9]{lrs-tri} for combinatorial properties
of~$C_d(n)$, in particular \defn{Gale's Evenness Criterion} which
characterizes the index sets of \defn{upper} and \defn{lower} facets
of~$C_d(n)$.

Cyclic polytopes yield our first examples of
PSN polytopes:

\begin{example}
  For any integers $k\ge0$ and $n\ge2k+2$, the cyclic polytope~${C_{2k+2}(n+1)}$ is $(k,n)$-PPSN.
\end{example}

\begin{example}
\label{ex:prodcyclic}
   For any $k\ge0$ and $\sub{n} \eqdef (n_1,\dots,n_r)$ with $r\ge1$ and $n_i\ge1$ for all $i$, define
   $I \eqdef \{i\in[r]\ | \  n_i\ge2k+3\}$. Then the product
\[
   \prod_{i\in I}C_{2k+2}(n_i+1)\times\prod_{i\notin I}
   \simplex_{n_i}
\]
is a $(k,\sub{n})$-PPSN polytope of dimension $(2k+2)|I|+\sum_{i\notin
  I} n_i$ (which is smaller than $\sum n_i$ when $I$ is
nonempty). Consequently,
\[
  \delta(k,\sub{n})
  \ \le \
  \delta_{pr}(k,\sub{n})
  \ \le \
  (2k+2)|I|+\sum_{i\notin I} n_i.
\]
\end{example}

\subsection{Reflections of cyclic polytopes}

Our next example deals with the special case of the product
$\simplex_1\times\simplex_n$ of a segment with a simplex. Using
products of cyclic polytopes as in Example~\ref{ex:prodcyclic}, we can
realize the $k$-skeleton of this polytope in dimension $2k+3$. We can
lower this dimension by $1$ by reflecting the cyclic polytope
$C_{2k+2}(n+1)$ through a well-chosen hyperplane:

\begin{proposition}
  \label{prop:reflect}
For any $k\ge0$, $n\ge2k+2$ and $\lambda\in\R$ sufficiently large, the polytope
\[
   P
   \ \eqdef \
   \conv\left(\big\{(t_i,\dots,t_i^{2k+2})^T\ | \ 
   i\in[n+1]\big\}
   \ \cup \
   \big\{(t_i,\dots,t_i^{2k+1},\lambda-t_i^{2k+2})^T\ | \ 
   i\in[n+1]\big\}\right)
\]
is a $(k,(1,n))$-PSN polytope of dimension $2k+2$.
\end{proposition}

\begin{proof}
  The polytope $P$ is obtained as the convex hull of two copies of the
  cyclic polytope $C_{2k+2}(n+1)$. The first one
  $Q \eqdef \conv\{\mu_{2k+2}(t_i)\ | \ i\in[n+1]\}$ lies on the moment
  curve~$\mu_{2k+2}$, while the second one is obtained as a reflection of $Q$ with
  respect to a hyperplane that is orthogonal to the last coordinate
  vector $u_{2k+2}$ and sufficiently far away. During this process,
\begin{enumerate} 
\item we destroy all the faces of $Q$ only contained in upper facets
  of $Q$;
\item we create prisms over faces of $Q$ that lie in at least one
  upper and one lower facet of~$Q$. In other words, we create prisms
  over the faces of~$Q$ strictly preserved under the orthogonal
  projection $\pi:\R^{2k+2}\to\R^{2k+1}$ with kernel $\R u_{2k+2}$.
\end{enumerate}

The projected polytope $\pi(Q)$ is nothing but the cyclic polytope
$C_{2k+1}(n+1)$. Since this polytope is $k$-neighborly, any face $F$
of dimension at most~$k-1$ in~$Q$ is strictly preserved
by~$\pi$. Thus, we take the prism over all faces of $Q$ of dimension
at most~$k-1$.

Thus, in order to complete the proof that the $k$-skeleton of $P$ is
that of $\simplex_1\times\simplex_n$, it is enough to show that any
$k$-face of $Q$ remains in $P$. This is obviously the case if this
$k$-face is also a $k$-face of $C_{2k+1}(n+1)$, and follows from the
next combinatorial lemma otherwise.
\end{proof}

\begin{lemma}
  A $k$-face of $C_{2k+2}(n+1)$ which is not a $k$-face of
  $C_{2k+1}(n+1)$ is only contained in lower facets of
  $C_{2k+2}(n+1)$.
\end{lemma}

\begin{proof}
  Let $F\subset[n+1]$ be a $k$-face of $C_{2k+2}(n+1)$. We assume that~$F$ is
  contained in at least one upper facet $G\subset[n+1]$
  of~$C_{2k+2}(n+1)$. Since the size of the final block of an upper facet of a cyclic polytope is odd,
  $G$ contains~$n+1$. If $n+1\in G\smallsetminus F$, then
  $G\smallsetminus\{n+1\}$ is a facet of $C_{2k+1}(n+1)$
  containing~$F$. Otherwise, $n+1\in F$, and
  $F' \eqdef F\smallsetminus\{n+1\}$ has only $k$ elements. Thus, $F'$~is a
  face of $C_{2k}(n)$, and can be completed to a facet
  of~$C_{2k}(n)$. Adding the index $n+1$ back to this facet, we obtain
  a facet of $C_{2k+1}(n+1)$ containing~$F$.  
  In both cases, we have shown that $F$ is a $k$-face of
  $C_{2k+1}(n+1)$.
\end{proof}

\subsection{Minkowski sums of cyclic polytopes}\label{subsec:minkowskiSumCyclicPolytopes}

Our next examples are Minkowski sums of cyclic polytopes. We first
describe an easy construction that avoids all technicalities, but only
yields $(k,\sub{n})$-PPSN polytopes in dimension $2k+2r$. After that,
we show how to reduce the dimension to $2k+r+1$, which according to
Corollary~\ref{cor:topObstr} is best possible for large~$n_i$'s.

\begin{proposition}\label{prop:firstMinkowskiSum}
  Let $k\ge0$ and $\sub{n} \eqdef (n_1,\dots,n_r)$ with $r\ge1$ and $n_i\ge1$ for all $i$.
  For any pairwise disjoint index sets $I_1,\dots, I_r\subset\R$, with $|I_i|=n_i$ for all~$i$, the
  polytope 
  \[
     P
     \ \eqdef \
     \conv\big\{v_{a_1,\dots,a_r}\  |\ (a_1,\dots,a_r)\in I_1\times\dots\times
     I_r\big\}
     \ \subset \
     \R^{2k+2r}
  \]
  is $(k,\sub{n})$-PPSN, where
  \[
    v_{a_1,\dots,a_r}
    \ \eqdef \
    \left(\sum_{i\in[r]} a_i,\sum_{i\in[r]}
      a_i^2,\dots,\sum_{i\in[r]} a_i^{2k+2r}\right)^T\in\R^{2k+2r}.
  \]
\end{proposition}

\begin{proof}
  The vertex set of  $\simplex_{\sub{n}}$ is
  indexed by $I_1\times\dots\times I_r$.  Let
  $A \eqdef A_1\times\dots\times A_r\subset I_1\times\dots\times I_r$
  define a $k$-face of $\simplex_{\sub{n}}$. Consider the polynomial
  \[
      f(t) 
      \ \eqdef \
      \prod_{i\in[r]}\prod_{a\in A_i} (t-a)^2
      \ = \
      \sum_{j=0}^{2k+2r} c_j t^j.
  \]
  Since $A$~indexes a $k$-face of~$\simplex_{\sub{n}}$, we know that $\sum
  |A_i|=k+r$, so that the degree of $f(t)$ is indeed $2k+2r$.
  Since $f(t)\ge0$, and equality holds if and only if
  $t\in\bigcup_{i\in[r]}
  A_i$, the inner product
  $\dotprod{(c_1,\dots,c_{2k+2r})}{v_{a_1,\dots,a_r}}$ equals
  \[
     (c_1,\dots, c_{2k+2r})
     \begin{pmatrix}
       \sum_{i\in[r]} a_i\\ \vdots \\ \sum_{i\in[r]} a_i^{2k+2r}
     \end{pmatrix}
     \ = \
     \sum_{i\in[r]}\sum_{j=1}^{2k+2r} c_j a_i^j
     \ = \
     \sum_{i\in[r]}
     \big(f(a_i)-c_0\big)
     \ \ge \
     -rc_0,
  \]
  with equality if and only if $(a_1,\dots,a_r)\in A$. Thus,
  $A$~indexes a face of~$P$ defined by the
  linear inequality $\sum_{i\in[r]} c_i x_i \ge -rc_0$.
  
  We thus obtain that the $k$-skeleton of~$P$ completely contains the $k$-skeleton of~$\simplex_{\sub{n}}$.
  Since~$P$ is furthermore a projection of~$\simplex_{\sub{n}}$, the faces of~$\simplex_{\sub{n}}$ are the only candidates
  to be faces of~$P$. We conclude that the $k$-skeleton of~$P$ is actually combinatorially equivalent to that of~$\simplex_{\sub{n}}$.
\end{proof}

To realize the $k$-skeleton of
$\simplex_{n_1}\times\dots\times\simplex_{n_r}$ even in
dimension~$2k+r+1$, we slightly modify this construction in the
following way.

\begin{theorem}
\label{theo:UBminkowskiCyclic}
Let $k\ge0$ and $\sub{n} \eqdef (n_1,\dots,n_r)$ with $r\ge1$ and $n_i\ge1$
for all $i$.  There exist pairwise disjoint index sets $I_1,\dots,
I_r\subset\R$, with $|I_i|=n_i$ for all~$i$, such that the polytope
  \[
     P
     \ \eqdef \ 
     \conv\{w_{a_1,\dots,a_r}\  |\ (a_1,\dots,a_r)\in I_1\times\dots\times
     I_r\} 
     \ \subset \ 
     \R^{2k+r+1}
  \]
  is $(k,\sub{n})$-PPSN, where
  \[
    w_{a_1,\dots,a_r}
    \ \eqdef \ 
    \bigg(a_1,\dots,a_r,\sum_{i\in[r]}
      a_i^2,\dots,\sum_{i\in[r]} a_i^{2k+2}\bigg)^T\in\R^{2k+r+1}.
  \]
\end{theorem}

\begin{proof}
  We will choose the index sets $I_1,\dots, I_r$ to be sufficiently
  separated in a sense that will be made explicit later in the
  proof. For each $k$-face~$F$ of $\simplex_{\sub{n}}$, indexed by
  $A_1\times\dots\times A_r\subset I_1\times\dots\times I_r$, our
  choice of the~$I_i$'s will ensure the existence of a monic
  polynomial
  \[
    f_F(t)
    \ \eqdef \
    \sum_{j=0}^{2k+2} c_jt^j, 
  \]
  which, for all $i\in[r]$, can be decomposed as 
  \[
    f_F(t)
    \ = \
   Q_i(t)\prod_{a\in A_i}(t-a)^2  +s_it+r_i,
 \]
 where $Q_i(t)$ is an everywhere positive polynomial of degree
 $2k+2-2|A_i|$, and $r_i,s_i\in\R$.  Assuming the existence of such a decomposable
 polynomial~$f_F$, we built from its coefficients the vector
\[
   n_F
   \ \eqdef \ 
   (s_1-c_1,\dots,s_r-c_1,-c_2,-c_3,\dots,-c_{2k+2})
   \ \in \ 
   \R^{2k+r+1},
\]
and prove that $n_F$ is normal to a supporting hyperplane
for~$F$. Indeed, for any $r$-tuple $(a_1,\dots,a_r)\in I_1\times\dots\times
I_r$, the inner product $\dotprod{n_F}{w_{a_1,\dots,a_r}}$ satisfies
the following inequality:
\begin{eqnarray*}
\dotprod{n_F}{w_{a_1,\dots,a_r}}
&=& \sum_{i\in[r]}\left(s_i a_i-\sum_{j=1}^{2k+2} c_ja_i^j\right)\\
&=& \sum_{i\in[r]}\left(s_ia_i+c_0-f_F(a_i)\right) \\
&=& \sum_{i\in[r]}\left(c_0-Q_i(a_i)\prod_{a\in A_i}(a_i-a)^2-r_i\right)
\ \le\  rc_0-\sum_{i\in[r]} r_i.
\end{eqnarray*}
Since the $Q_i$'s are everywhere positive, equality holds if and only
if ${(a_1,\dots,a_r)\in A_1\times\dots\times A_r}$. Given the existence
of a decomposable polynomial~$f_F$, this proves that
$A_1\times\dots\times A_r$ indexes all $w_{a_1,\dots,a_r}$'s that lie
on a face $F'$ in $P$, and they of course span $F'$ by definition of
$P$. To prove that $F'$ is combinatorially equivalent to $F$, it
suffices to show that each $w_{a_1,\dots,a_r}\in F'$ is in fact a
vertex of $P$, since $P$ is a projection of $\simplex_{\sub{n}}$. This
can be shown with the normal vector $(2a_1,\dots,2a_r,-1,0,\dots,0)$,
using the same calculation as before.

As in the proof of Proposition~\ref{prop:firstMinkowskiSum}, this ensures that the
$k$-skeleton of~$P$ completely contains the $k$-skeleton of~$\simplex_{\sub{n}}$,
and we argue that they actually coincide since~$P$ is furthermore a projection of~$\simplex_{\sub{n}}$.

Before showing how to choose the index sets $I_i$ that enable us to
construct the polynomials~$f_F$ in general, we illustrate the proof on
the smallest example.
\end{proof}

\begin{example}
  Let $k=1$ and $\sub{n} \eqdef (n_1,n_2)$. Choose the index sets $I_1$, $I_2\subset\R$ with $|I_1|=n_1$, $|I_2|=n_2$ and 
  separated in the sense that the largest element of~$I_1$ be smaller
  than the smallest element of~$I_2$. For any $1$-dimensional face $F$
  of~$\simplex_{\sub{n}}$ indexed by $\{a,b\}\times\{c\}\subset I_1\times I_2$,
  consider the polynomial $f_F$ of degree $2k+2=4$:
  \[
    f_F(t)
    \ \eqdef \
    (t-a)^2(t-b)^2
    \ = \
    (t^2+\alpha t+\beta)(t-c)^2 + s_2t + r_2,
  \]
  where
  \begin{eqnarray*}
    \alpha & = & 2(-a-b+c), \\ 
    \beta  & = & a^2 + b^2 + 3c^2 + 4ab - 4ac - 4bc, \\
    r_2    & = &  a^2b^2 - \beta c^2, \\
    s_2    & = &  -2a^2b - 2ab^2 - \alpha c^2 + 2\beta c.
  \end{eqnarray*}
  Since the index sets $I_1$, $I_2$ are separated, the discriminant
  $\alpha^2-4\beta = -8(c-a)(c-b)$ is negative, which implies that the
  polynomial $Q_2(t)=t^2+\alpha t+\beta$ is positive for all values
  of~$t$. A symmetric formula holds for the $1$-dimensional faces of
  $\simplex_{\sub{n}}$ whose index sets are of the form $\{a\}\times\{b,c\} \subset I_1 \times I_2$.
\end{example}

\begin{proof}[Proof of Theorem~\ref{theo:UBminkowskiCyclic}, continued]
  We still need to show how to choose the index sets $I_i$ that enable
  us to construct the polynomials~$f_F$ in general.  Once we have
  chosen these index sets, finding~$f_F$ is equivalent to the task of
  finding polynomials~$Q_i(t)$ such that

\begin{enumerate}[(i)]
\item $Q_i(t)$ is monic of degree $2k+2-2|A_i|$.
\item The $r$ polynomials $f_i(t) \eqdef Q_i(t)\prod_{a\in A_i}(t-a)^2$ are
  equal, up to the coefficients on $t^0$ and~$t^1$.
\item $Q_i(t)>0$ for all $t\in\R$.
\end{enumerate}

The first two items form a linear system equations on the
coefficients of the $Q_i(t)$'s which has the same number of equations
as variables, namely~$2k(r-1)$. We show that it has a unique solution if one
chooses the correct index sets~$I_i$, and we postpone the discussion of requirement
(iii) to the end of the proof. To do this, choose distinct reals $\bar
a_1,\dots,\bar a_r\in\R$ and look at the similar equation system:

\begin{enumerate}[(i)]
\item $\bar Q_i(t)$ are monic polynomials of degree $2k+2-2|A_i|$.
\item The $r$ polynomials $\bar f_i(t) \eqdef \bar Q_i(t)(t-\bar
  a_i)^{2|A_i|}$ are equal, up to the coefficients on $t^0$ and~$t^1$.
\end{enumerate}

The first equation system moves into the second when we deform the points of the sets $A_i$ continuously to $\bar a_i$, respectively. By continuity of the determinant, if the second equation system has a unique solution then so has the first equation system as long as we chose the sets $I_i$ close enough to the $\bar a_i$'s for all~$i$. Observe that in the end, we can fulfill all these closeness conditions required for all $k$-faces of $\simplex_{\sub{n}}$ since there are only finitely many $k$-faces.

Note that a polynomial $\bar f_i(t)$ of degree $2k+2$ has the form 
\begin{equation}
\label{eqfstarhasform}
\bar Q_i(t)(t-\bar a_i)^{2|A_i|}+ s_it+r_i,
\end{equation}
for a monic polynomial $\bar Q_i$ and some reals $s_i$ and $r_i$ if and only if $\bar f''_i(t)$ has the form
\begin{equation}
\label{eqfprimeprimestarhasform}
R_i(t)(t-\bar a_i)^{2(|A_i|-1)},
\end{equation}
for some polynomial $R_i(t)$ with leading coefficient
$(2k+2)(2k+1)$. The backward direction can be settled by assuming,
without loss of generality, that $\bar a_i=0$. Indeed, otherwise make
a change of variables $(t-\bar a_i)\mapsto t$ and then integrate
\eqref{eqfprimeprimestarhasform} twice (with constants of integration
equal to zero) to obtain \eqref{eqfstarhasform}.

Therefore the second equation system is equivalent to the following third one:

\begin{enumerate}[(i)]
\item $R_i(t)$ are polynomials of degree $2k-2(|A_i|-1)$ with leading coefficient $(2k+2)(2k+1)$.
\item The $r$ polynomials $g_i(t) \eqdef R_i(t)(t-\bar a_i)^{2(|A_i|-1)}$ all equal the same polynomial, say~$g(t)$.
\end{enumerate}

Since $\sum_i 2(|A_i|-1)=2k$, this system of equations has the unique
solution
$$R_i(t) = (2k+2)(2k+1)\prod_{j\neq i} (t-\bar a_j)^{2(|A_j|-1)},$$
with
$$g(t) = (2k+2)(2k+1)\prod_{j\in[r]} (t-\bar a_j)^{2(|A_j|-1)}.$$

Therefore, the first two systems of equations both have a unique solution
(as long as the $I_i$'s are chosen sufficiently close to the $a_i$'s).
It thus only remains to deal with the positivity requirement (iii).

In the unique solution of the second equation system, the polynomial
$\bar f_i(t)$ is obtained by integrating $g_i(t)$ twice with some
specific integration constants. For a fixed $i$, we can again assume
$\bar a_i=0$. Then both integration constants were chosen to be zero
for this~$i$, hence $\bar f_i(0)=\bar f'_i(0)=0$. Since $g_i$ is
non-negative and zero only at isolated points, $\bar f_i$ is strictly
convex, hence non-negative and zero only at $t=0$. Therefore $\bar
Q_i(t)$ is positive for $t\neq 0$. Since we chose $\bar a_i=0$, we can
quickly compute the correspondence between the coefficients of $\bar
Q_i(t)=\sum_j \bar q_{i,j}t^j$ and of $R_i(t)=\sum_j r_{i,j}t^j$:
$$r_{i,j} = \big(2|A_i|(2|A_i|-1)+4j|A_i|+j(j-1)\big)\bar q_{i,j}.$$
In particular,
$$\bar Q_i(0)=\bar q_{i,0}=\frac{r_{i,0}}{2|A_i|(2|A_i|-1)}=\frac{R_i(0)}{2|A_i|(2|A_i|-1)}>0,$$
therefore $\bar Q_i(t)$ is everywhere positive.
Since the solutions of linear equation systems move continuously when one deforms the entries of the equation system by a homotopy,
this ensures that $Q_i(t)$ is everywhere positive if $I_i$ is chosen close enough to $\bar a_i$. The positivity of $Q_i(t)$ finishes the proof.
\end{proof}

\section{Projections of deformed products of simple polytopes}

In the previous section, we saw an explicit construction of polytopes
whose $k$-skeleton is equivalent to that of a product of simplices. In
this section, we provide another construction of $(k,\sub{n})$-PPSN
polytopes, using Sanyal \& Ziegler's technique of ``projecting
deformed products of polygons''~\cite{z-ppp-04,sz-capdd} and
generalizing it to products of arbitrary simple polytopes. This generalized technique consists in
projecting a suitable polytope that is combinatorially
equivalent to a given product of simple polytopes in such a way as to
preserve its complete $k$-skeleton. The special case of products of
simplices then yields $(k,\sub{n})$-PPSN polytopes.

\subsection{General situation}
\label{sec:general}

We first discuss the general setting: given a product
$P \eqdef P_1\times\dots\times P_r$ of simple polytopes, we construct a
polytope $\defP$ that is combinatorially equivalent to~$P$ and whose
$k$-skeleton is preserved under the projection onto the first $d$
coordinates.

\subsubsection{Deformed products of simple polytopes}

Let $P_1,\dots,P_r$ be \defn{simple} polytopes of respective
dimensions $n_1,\dots, n_r$ and facet descriptions $P_i =
\{x\in\R^{n_i}\ | \ A_i x\le b_i\}$. Here, each matrix
$A_i\in\R^{m_i\times n_i}$~has one row for each of the $m_i$~facets
of~$P_i$, and $b_i\in\R^{m_i}$. The product $P \eqdef P_1\times\dots\times
P_r$ then has dimension $n \eqdef \sum_{i\in[r]} n_i$, and its facet
description is given by the $m \eqdef \sum_{i\in[r]} m_i$ inequalities
\[
\begin{pmatrix}
A_1 & & \\
& \ddots & \\
& & A_r
\end{pmatrix}
x \ \le \
\begin{pmatrix}
b_1 \\
\vdots \\
b_r
\end{pmatrix}.
\]
The left hand $m\times n$ matrix, whose blank entries are all zero,
shall be denoted by~$A$. It is proved in \cite{az-dpmsp-99} that for
any matrix~$\defA$ obtained from~$A$ by \defn{arbitrarily} changing
the zero entries above the diagonal blocks, there exists a right-hand
side~$\defb$ such that the deformed polytope~$\defP$ defined by the
inequality system $\defA x\le \defb$ is combinatorially equivalent
to~$P$. The equivalence is the obvious one: it maps the facet defined
by the $i$-th row of~$A$ to the one given by the $i$-th row
of~$\defA$, for all~$i$. Following~\cite{sz-capdd}, we will use this
``deformed product'' construction in such a way that the projection
of~$\defP$ to the first $d$~coordinates preserves its $k$-skeleton in
the following sense.

\subsubsection{Preserved faces and the Projection Lemma}

For integers $n>d$, let $\pi:\R^n\to\R^d$ denote the orthogonal
projection to the first $d$ coordinates, and $\tau:\R^n\to\R^{n-d}$
denote the dual orthogonal projection to the last $n-d$ coordinates.
Let $P$ be a full-dimensional simple polytope in $\R^n$, with $0$ in
its interior.  The following notion of preserved faces ~---~see
Figure~\ref{fig:projection}~---~will be used extensively at the end of
this paper:

\begin{definition}[\cite{z-ppp-04}]
\label{def:spf}
  A proper face $F$ of a polytope $P$ is \defn{strictly preserved}
  under  $\pi$ if
\begin{enumerate}[(i)]
\item $\pi(F)$ is a face of $\pi(P)$,
\item $F$ and $\pi(F)$ are combinatorially isomorphic, and
\item $\pi^{-1}(\pi(F))$ equals $F$.
\end{enumerate}
\end{definition}

\begin{figure}[htbp]
   \centerline{\includegraphics[scale=1]{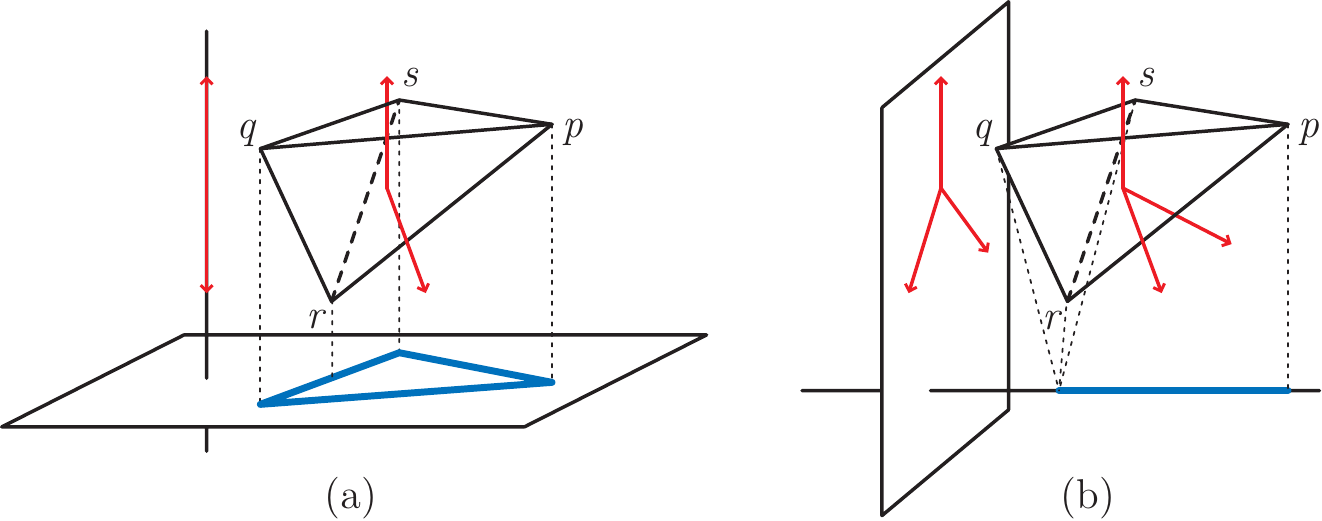}}
   \caption{{(a) Projection of a tetrahedron onto $\R^2$: the edge
    $pq$ is strictly preserved, while neither the edge $qr$, nor the
    face $\textit{qrs}$, nor the edge $\textit{qs}$ are (because of
    conditions (i), (ii) and (iii) respectively). (b) Projection of a
    tetrahedron to $\R$: only the vertex~$p$ is strictly
    preserved.}}
    \label{fig:projection}
\end{figure}

The characterization of strictly preserved faces of $P$ uses the
normal vectors of the facets of $P$. Let $F_1,\dots,F_m$ denote the
facets of $P$. For all $i\in[m]$, let $f_i$ denote the normal vector
to $F_i$, and let $g_i \eqdef \tau(f_i)$. For any face $F$ of $P$, let
$\varphi(F)$ denote the set of indices of the facets of~$P$ containing
$F$, \ie such that $F=\bigcap_{i\in \varphi(F)} F_i$.

\begin{lemma}[Projection Lemma~\cite{az-dpmsp-99,z-ppp-04}]\label{lem:projection}
A face $F$ of the polytope $P$ is strictly preserved under the
projection $\pi$ if and only if $\{g_i\ | \  i\in \varphi(F)\}$ is positively
spanning. \hfill$\Box$
\end{lemma}

\subsubsection{A first construction}

Let $t\in\{0,1,\dots,r\}$ be maximal such that the matrices
$A_1,\dots, A_t$ are entirely contained in the first $d$~columns
of~$A$. Let $\bm \eqdef \sum_{i=1}^t m_i$ and $\bn \eqdef \sum_{i=1}^t n_i$.  By
changing bases appropriately, we can assume that the bottom $n_i\times
n_i$ block of $A_i$ is the identity matrix for each~$i\ge t+1$.  In
order to simplify the exposition, we also assume first that $\bn=d$,
\ie that the projection on the first $d$ coordinates separates the
first $t$ block matrices from the last $r-t$. See
Figure~\ref{fig:defA}a.

Let $\{g_1,\dots, g_\bm\}\subset\R^{n-d}$ be a set of vectors such
that $G \eqdef \{e_1,\dots,e_{n-d}\}\cup\{g_1,\dots,g_\bm\}$~is the
Gale transform of a full-dimensional simplicial neighborly
polytope~$Q$~---~see \cite{z-lp-95,m-ldg-02} for definition and properties of Gale duality. By elementary properties of the Gale transform, $Q$~has
${\bm+n-d}$~vertices, and $\dim Q=(\bm+n-d)-(n-d)-1=\bm-1$.  In
particular, every subset of $\lfloor\frac{\bm-1}{2}\rfloor$~vertices
spans a face of~$Q$, so every subset of
$\bm+n-d-\lfloor\frac{\bm-1}{2}\rfloor \eqfed \alpha$ elements of $G$ is
positively spanning.

We deform the matrix $A$ into the matrix~$\defA$ of~Figure~\ref{fig:defA}a,
using the vectors~$g_1,\dots,g_\bm$ to deform the top $\bm$ rows. We denote by $\defP$ the corresponding deformed product. We say that a facet of
$\defP$ is ``good'' if the right part of the corresponding row of $\defA$ is covered by
a vector of $G$, and ``bad'' otherwise. Bad facets are hatched in
Figure~\ref{fig:defA}a.  Observe that there are $\beta \eqdef m-\bm-n+d$ bad
facets in total.

\begin{figure}
  \centerline{\includegraphics[scale=1]{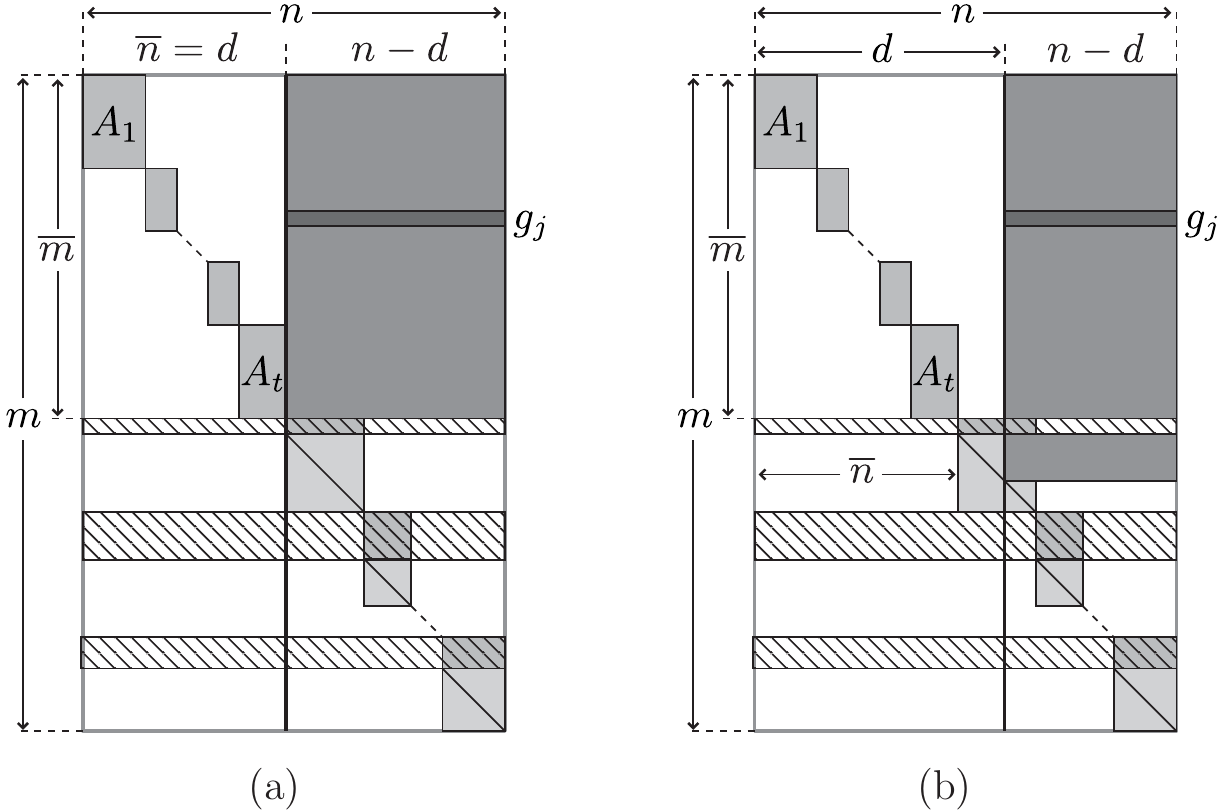}}
  \caption{The deformed matrix $\defA$ (a) when the projection does
    not slice any block ($\bn=d$), and (b) when the block $A_{t+1}$ is
    sliced ($\bn<d$).  Horizontal hatched boxes denote bad row
    vectors. The top right solid block is formed by the vectors
    $g_1,\dots,g_\bm$. }
  \label{fig:defA}
\end{figure}

Let $F$ be a $k$-face of~$\defP$. Since $\defP$ is a simple
$n$-dimensional polytope, $F$~is the intersection of $n-k$~facets,
among which at least $\gamma \eqdef n - k - \beta$ are good facets.  If the
corresponding elements of~$G$ are positively spanning, then $F$~is
strictly preserved under projection onto the first $d$~coordinates.
Since we have seen that any subset of $\alpha$~vectors of~$G$ is
positively spanning, $F$~will surely be preserved if $\alpha\le\gamma$,
which is equivalent to
\[
	k
	\ \le \
	n - m  + \fracfloor{\bm-1}{2}.
\]
Thus, under this assumption, we obtain a $d$-dimensional polytope
whose $k$-skeleton is combinatorially equivalent to that of
$P \eqdef P_1\times\dots\times P_r$.

\subsubsection{When the projection slices a block}

We now discuss the case when $\bn < d$, for which the method is very
similar.  We consider vectors $g_1,\dots,g_{\bm + d - \bn}$ such that
$G \eqdef \{e_1,\dots,e_{n-d}\}\cup\{g_1,\dots,g_{\bm + d -\bn}\}$ is the
Gale dual of a neighborly polytope. We deform the
matrix $A$ into the matrix $\defA$ shown in
Figure~\ref{fig:defA}b, using again the vectors $g_1,\dots,g_\bm$ to deform the
top $\bm$ rows and the vectors $g_{\bm+1}\dots,g_{\bm + d -\bn}$ to deform the top $d-\bn$ rows of the $n_{t+1}\times n_{t+1}$ bottom identity submatrix of
$A_{t+1}$. This is indeed a valid deformation since we
can prescribe the $n_{t+1}\times n_{t+1}$ bottom submatrix of
$A_{t+1}$ to be any upper triangular matrix, up to changing the basis
appropriately.  For the same reasons as before,
\begin{enumerate}
\item any subset of at least $\alpha \eqdef \bm + n - \bn - \fracfloor{\bm
    + d - \bn - 1}{2}$ elements of $G$ is positively spanning;
\item the number of bad facets is $\beta \eqdef m-\bm-n+\bn$, and thus any
  $k$-face of $\defP$ is contained in at least $\gamma \eqdef n-k-\beta$
  good facets.
\end{enumerate}
Thus, the condition $\alpha\le\gamma$ translates to
\[
	k
	\ \le \
	n - m  + \fracfloor{\bm+d-\bn-1}{2},
\]
and we obtain the following proposition.

\begin{proposition}\label{prop:defp1}
  Let $P_1,\dots,P_r$ be simple polytopes of respective dimension
  $n_i$, and with $m_i$ many facets. For a fixed integer $d\le
  \sum_{i=1}^r n_i$, let $t$~be maximal such that $\sum_{i=1}^t
  n_i\le d$. Then there exists a $d$-dimensional polytope whose
  $k$-skeleton is combinatorially equivalent to that of the product
  $P_1\times\dots\times P_r$, provided
\[
  0 \ \le \ k \ \le \ \sum_{i=1}^r (n_i-m_i) + \Lfloor\frac{1}{2}\left(d-1+\sum_{i=1}^t(m_i-n_i)\right)\Rfloor. \qed
\]
\end{proposition}

In the next two paragraphs, we present two improvements on the bound of this proposition. Both use colorings of the graphs of the polar polytopes $P_i^\polar$, in order to weaken the condition $\alpha\le\gamma$, in two different directions:
\begin{enumerate}[(i)]
\item the first improvement decreases the number of required vectors
  in the Gale transform~$G$, which, in turn, decreases the value of
  $\alpha$;
\item the second one decreases the number of bad facets, and thus
  increases the value of $\gamma$.
\end{enumerate}

\subsubsection{Multiple vectors}

In order to raise our bound on $k$, we can save vectors of $G$ by
repeating some of them several times. Namely, any two facets that have
no $k$-face in common can share the same vector $g_j$.  Since any two
facets of a simple polytope containing a common $k$-face share a
ridge, this condition can be expressed in terms of incidences in the
graph of the polar polytope: facets not connected by an edge in this
graph can use the same vector~$g_j$. We denote the chromatic number of
a graph~$H$ by~$\chi(H)$.  Then, each $P_i$ with $i\le t$ only
contributes $\chi_i \eqdef \chi(\text{sk}_1 P_i^\polar)$ different vectors
in~$G$, instead of~$m_i$ of them. Thus, we only need in total
$\bchi \eqdef \sum_{i=1}^t \chi_i$ different vectors $g_j$. This improvement
replaces $\bm$ by $\bchi$ in the formula of $\alpha$, while
$\beta$~and~$\gamma$ do not change, and the condition
$\alpha\le\gamma$ is equivalent to
\[
	k
	\ \le \
	n - m + \bm - \bchi + \fracfloor{\bchi-d-\bn-1}{2}.
\]
Thus, we obtain the following improved proposition:

\begin{proposition}\label{prop:defp2}
  Let $P_1,\dots,P_r$ be simple polytopes of respective
  dimension~$n_i$, and with $m_i$~many
  facets. Let~$\chi_i \eqdef \chi(\text{sk}_1 P_i^\polar)$ denote the
  chromatic number of the graph of the polar polytope
  $P_i^\polar$. For a fixed integer $d\le\sum_{i=1}^r n_i$, let
  $t$~be maximal such that $\sum_{i=1}^t n_i\le d$. Then there exists
  a $d$-dimensional polytope whose $k$-skeleton is combinatorially
  equivalent to that of the product $P_1\times\dots\times P_r$,
  provided
\[
  0 \ \le \ k \ \le \ \sum_{i=1}^r (n_i-m_i) + \sum_{i=1}^t (m_i-\chi_i)+
  \Lfloor\frac12\left(d-1+\sum_{i=1}^t(\chi_i-n_i)\right)\Rfloor. \qed
\]
\end{proposition}

\begin{example}\label{ex:even}
  Since polars of simple polytopes are simplicial, $\chi_i\ge n_i$ is
  an obvious lower bound for the chromatic number of the dual graph
  of~$P_i$. Polytopes that attain this lower bound with equality are
  characterized by the property that all their $2$-dimensional faces
  have an even number of vertices, and are called \defn{even}
  polytopes. 

  If all $P_i$ are even polytopes, then $\bn=\bchi$, and we obtain a $d$-dimensional polytope with the same $k$-skeleton as $P_1\times\dots\times P_r$ provided
  \[
     k
     \ \le \
     n - m + \bm - \bn + \fracfloor{d-1}{2}.
  \]
  In order to maximize $k$, we should maximize $\bm-\bn$, subject to
  the condition $\bn\le d$. For example, if all $n_i$ are equal, this
  amounts to ordering the $P_i$ by decreasing number of facets.
\end{example}

\subsubsection{Scaling blocks}

We can also apply colorings to the blocks $A_i$ with $i\ge t+1$, by
filling in the area below~$G$ and above the diagonal blocks. To
explain this, assume for the moment that $\chi_i\le n_{i+1}$ for a
certain fixed~$i\ge t+2$.  Assume that the rows of $A_i$ are colored
with $\chi_i$~colors using a valid coloring $c:[m_i]\to[\chi_i]$
of the graph of the polar polytope~$P_i^\polar$. Let $\Gamma$ be the
incidence matrix of $c$, defined by $\Gamma_{j,k}=1$
if~$c(j)=k$, and $\Gamma_{j,k}=0$ otherwise. Thus, $\Gamma$ is a
matrix of size~$m_i\times \chi_i$. We put this matrix to the right
of~$A_i$ and above~$A_{i+1}$ as in Figure~\ref{fig:chi}b, so that we
append the same unit vector to each row of~$A_i$ in the same color
class. Moreover, we scale all entries of the block~$A_i$ by a
sufficiently small constant $\varepsilon>0$.

\begin{figure}
  \centerline{\includegraphics[scale=1]{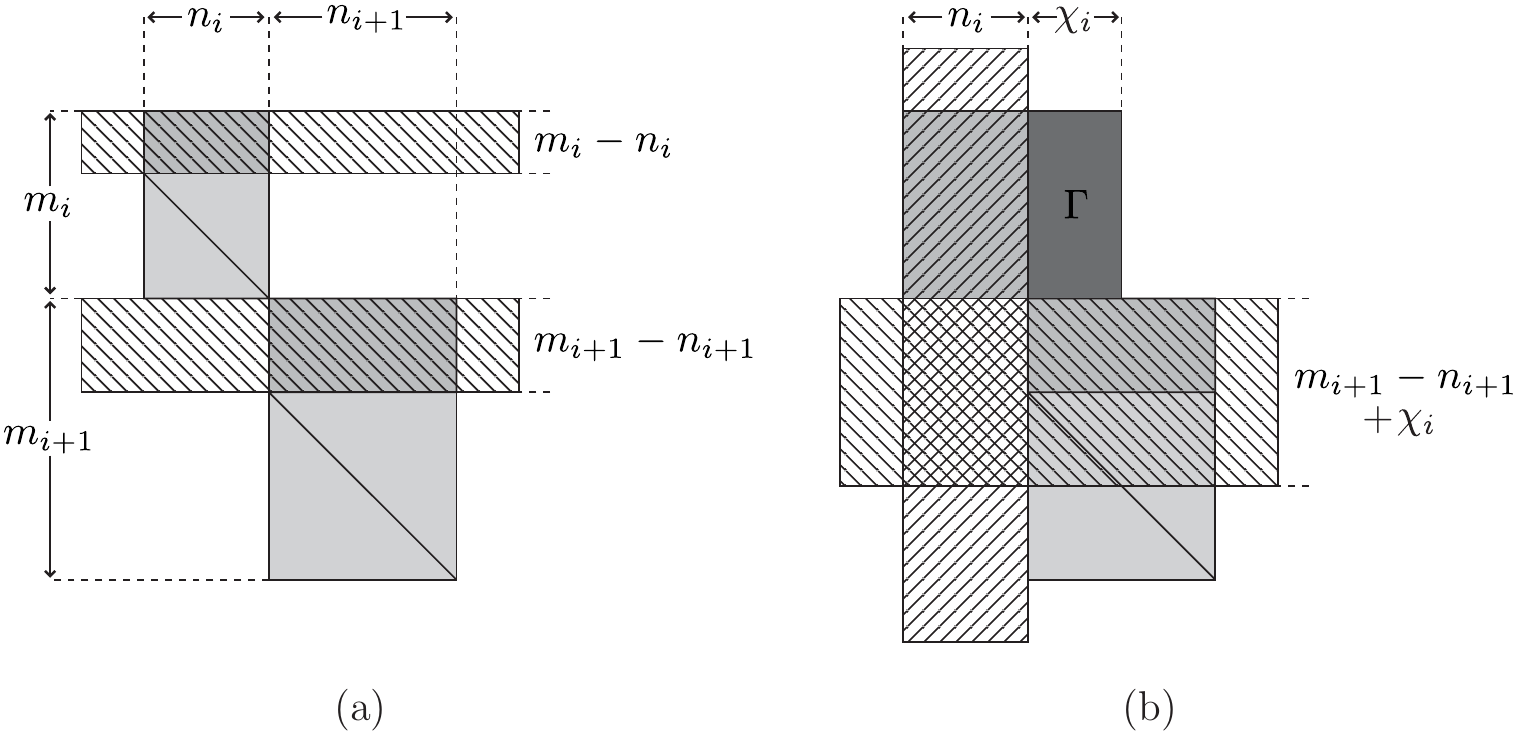}}
  \caption{How to raise the dimension of the preserved skeleton by
    inserting the incidence matrix~$\Gamma$ of a coloring of the graph
    of the polar polytope~$P_i^\polar$. Part~(a) shows the situation
    before the insertion of~$\Gamma$, and part~(b) the changes that
    have occurred.  Bad row vectors and unnecessary columns are 
    hatched. The entries in the matrix to the left of~$\Gamma$ must be
    rescaled to retain a valid inequality description of~$P$.}
  \label{fig:chi}
\end{figure}

In this setting, the situation is slightly different:
\begin{enumerate}
\item In the Gale dual $G$, we do not need the $n_i$ basis vectors
of $\R^{n-d}$ hatched in Figure~\ref{fig:chi}b. 
Let $a \eqdef \sum_{j<i} n_j$ denote the index of the last column vector
of $A_{i-1}$ and $b \eqdef 1+\sum_{j\le i} n_j$ denote the index
of the first column vector of $A_{i+1}$.
We define $G \eqdef \{e_1,\dots,e_{a-d},e_{b-d},\dots,e_{n-d}\}\cup
\{g_1,\dots,g_\bm\}$ to be the Gale transform of a simplicial
neighborly polytope~$Q$ of dimension $\bm-1-n_i$. As before, any
subset of $\alpha \eqdef \bm+n-\bn-n_i-\lfloor\frac{\bm+d-\bn-n_i-1}{2}\rfloor$ vectors of
$G$ positively spans $\R^{n-d}$.

\item ``Bad'' facets are defined as before, except that the top
  $m_i-n_i$ rows of~$A_i$ are not bad anymore, but all of the first
  $m_{i+1}-n_{i+1}+\chi_i$~rows of~$A_{i+1}$ are now bad.  Thus, the
  net change in the number of bad rows is $\chi_i-m_i+n_i$, so that
  any $k$-face is contained in at least
  $\gamma \eqdef 2n-k-m+\bm-\bn+m_i-n_i-\chi_i$ good rows. Up to
  $\varepsilon$-entry elements, the last $n-d$ coordinates of these
  rows correspond to pairwise distinct elements of~$G$.
\end{enumerate}

Applying the same reasoning as above, the $k$-skeleton of~$\defP$ is
strictly preserved under projection to the first~$d$ coordinates as
soon as $\alpha\le\gamma$, which is equivalent to
\[
	k
	\ \le \
	n - m + m_i -\chi_i + \fracfloor{\bm+d-\bn-n_i-1}{2}.
\]

Thus, we improve our bound on $k$ provided
\[
\Delta \eqdef m_i-\chi_i+\fracfloor{\bm+d-\bn-n_i-1}{2} -
\fracfloor{\bm+d-\bn-1}{2}>0.
\]
For example, this difference
$\Delta$ is big for polytopes whose polars have many vertices but
a small chromatic number.

\medskip

Finally, observe that one can apply this ``scaling''
improvement even if $\chi_i>n_{i+1}$ (except that it will perturb more
than the two blocks $A_i$ and $A_{i+1}$) and to more than one
matrix~$A_i$. Please see the example in
Figure~\ref{fig:matrixfinal}. In this picture, the $\Gamma$~blocks are
incidence matrices of colorings of the graphs of the polar
polytopes. Call ``diagonal entries'' all entries on the diagonal of
the $n_i \times n_i$ bottom submatrix of a factor $A_i$. A column is
unnecessary (hatched in the picture) if its diagonal entry has a
$\Gamma$ block on the right and no $\Gamma$ block above. Good rows are
those covered by a vector $g_j$ or a $\Gamma$ block, together with the basis
vectors whose diagonal entry has no $\Gamma$ block above (bad rows are
hatched in the picture).

\begin{figure}[htbp]
  \centerline{\includegraphics[scale=1]{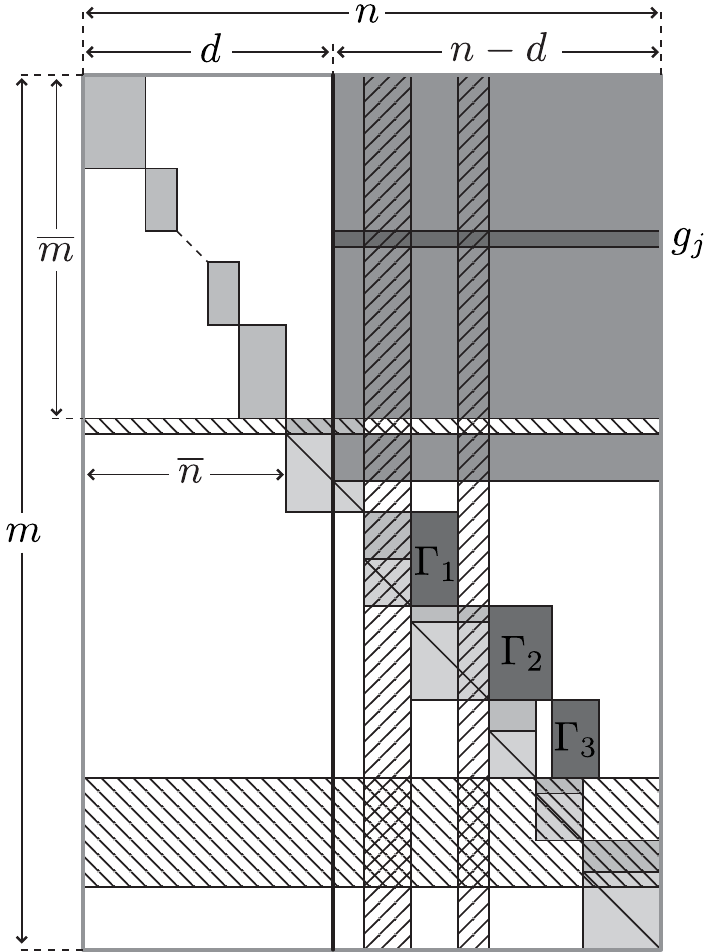}}
  \caption{How to reduce the number of vectors in the Gale transform
    using various coloring matrices of polar polytopes. Situations
    where $\chi_i > n_{i+1}$ can be accommodated for as illustrated by
    the matrix $\Gamma_2$ in the picture.}
  \label{fig:matrixfinal}
\end{figure}

\begin{example}
  \begin{enumerate}
  \item If $P_i$ is a segment, then $n_i=1$, $m_i=2$ and $\chi_i=1$,
    so that $\Delta=1$ if $\bm$~is even and $0$~otherwise. Iterating
    this, if $P_i$~is an $s$-dimensional cube, then
    $\Delta\simeq\frac{s}{2}$. This yields \defn{neighborly
      cubical polytopes}~---~see~\cite{jz-ncp-00, js-ncps-05}.
  \item If $P_i$ is an even cycle, then $n_i=2$, $m_i=2p$ and
    $\chi_i=2$, so that $\Delta=2p-3$. This yields \defn{projected
      products of polygons}~---~see \cite{z-ppp-04,sz-capdd}.
  \end{enumerate}
\end{example}

In general, it is difficult to give the explicit ordering of the
factors and choice of deformation that will yield the largest possible
value of~$k$ attainable by a concrete product $P_1\times\dots\times
P_r$ of simple polytopes, and consequently to summarize this improvement
by a precise proposition as we did for our first improvement.
However, this best value can clearly be
found by optimizing over the finite set of all possible orderings and
types of deformation. Furthermore, we can be much more explicit for
products of simplices, as we detail in the next section.

\subsection{Projection of deformed product of simplices}

We are now ready to apply this general construction to the particular
case of products of simplices. For this, we represent the simplex
$\simplex_{n_i}$ by the inequality system $A_i x \le b_i$, where
\[
   A_i 
   \ \eqdef \
   \begin{pmatrix}
     -1 & \dots & -1 \\
     1 \\
     & \ddots\\
     && 1
   \end{pmatrix}
\]
and $b_i$ is a suitable right-hand side. We express the results of the construction with a case distinction according to the number $s \eqdef |\{i\in[r]\ |\ n_i=1\}|$ of segments in the product $\simplex_{\sub{n}}$.

\begin{proposition}\label{prop:defp-ppsn}
  Let $\sub{n} \eqdef (n_1,\dots,n_r)$ with $1=n_1=\dots=n_s<n_{s+1}\le\dots\le n_r$. Then
  \begin{enumerate}
  \item for any $0\le d\le s-1$, there exists a $d$-dimensional $(k,\sub{n})$-PPSN polytope provided
  \[
  k\ \le \ \fracfloor{d}{2}-r+s-1.
  \]
  \item for any $s\le d\le n$, there exists a $d$-dimensional $(k,\sub{n})$-PPSN polytope provided
  \[
  k\ \le \ \fracfloor{d+t-s}{2}-r+s.
  \]
  where $t\in\{s,\dots,r\}$ denotes the maximal integer such that $\sum_{i=1}^{t} n_i\le d$.
  \end{enumerate}
\end{proposition}

\begin{figure}[tbp]
  \centerline{\includegraphics[scale=1]{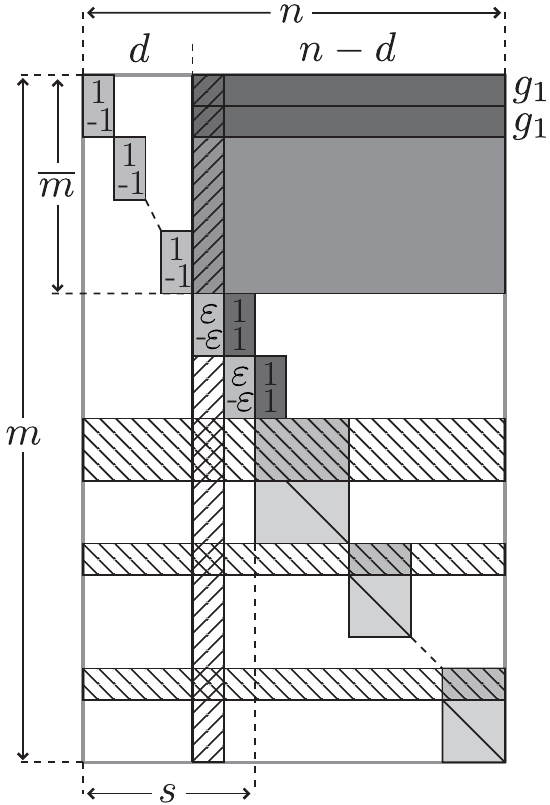}}
  \caption{How to obtain PPSN polytopes from a deformed product
    construction, when the number~$s$ of segment factors exceeds the
    target dimension~$d$ of the projection.}
  \label{fig:defp-ppsn1}
\end{figure}

\begin{proof}[Proof of (1)]
  This is a special case of the results obtainable with the methods of
  Section~\ref{sec:general}. The best construction is obtained using
  the matrix  in Figure~\ref{fig:defp-ppsn1}, from which we
  read off that 
  \begin{enumerate}
  \item any subset of at least $\alpha \eqdef n-\fracfloor{d}{2}$
    vectors in~$G$ is positively spanning; and
  \item the number of bad facets is $\beta \eqdef r-s+1$, and therefore any
    $k$-face of $\defP$ is contained in at least $\gamma \eqdef n-k-r+s-1$
    good facets.
  \end{enumerate}
  From this, the claim follows.
\end{proof}

\begin{proof}[Proof of (2)]
  Consider the deformed product of Figure~\ref{fig:defp-ppsn2}a.
  Using similar calculations as before, we deduce that
  \begin{enumerate}
  \item any subset of at least
    $\alpha \eqdef t-s+n-\fracfloor{d+t-s-1}{2}$ vectors in~$G$ is
    positively spanning; and
  \item the number of bad facets is $\beta \eqdef r-t$, and therefore any
    $k$-face of $\defP$ is contained in at least $\gamma \eqdef n-k-r+t$ good
    facets.
  \end{enumerate}
  This yields a bound of
    \[
  k\ \le \ \fracfloor{d+t-s-1}{2}-r+s.
  \]

  We optimize the final `$-1$' away by suitably deforming the matrix
  $A_{t+1}$ as in Figure~\ref{fig:defp-ppsn2}b. This amounts to adding
  one more vector $g_\star$ to the Gale diagram, so that the first row
  of~$A_{t+1}$ ceases to be a bad facet. This deformation is valid because:
  \begin{enumerate}

  \item the matrix 
    \[
    \begin{pmatrix}
      -1 & \dots & -1 & \star & \dots &  \star\\
      M \\
      & \ddots\\
      && M\\
      &&& 1\\
      &&&& \ddots\\
      &&&&& 1
    \end{pmatrix}
    \]
    still defines a simplex, as long as the `$\star$' entries are
    negative and $M\gg0$ is chosen to be sufficiently large;
 
  \item we can in fact choose the new vector $g_\star$ to have only
    negative entries, by imposing an additional restriction on the
    Gale diagram $G=\{e_1,\dots, e_{n-d}$, $g_1, \dots, g_{d+t},
    g_\star\}$ of~$Q$. Namely, we require that the vertices of the
    $(d+t)$-dimensional simplicial polytope~$Q$ that correspond to the
    Gale vectors $g_1,\dots,g_{d+t}$ lie on a facet. This forces
    the remaining vectors $e_1,\dots,e_{n-d},g_\star$ to be positively
    spanning, so that $g_\star$ has only negative entries. \qedhere
  \end{enumerate}
\end{proof}

\begin{figure}[tbp]
  \centerline{\includegraphics[scale=1]{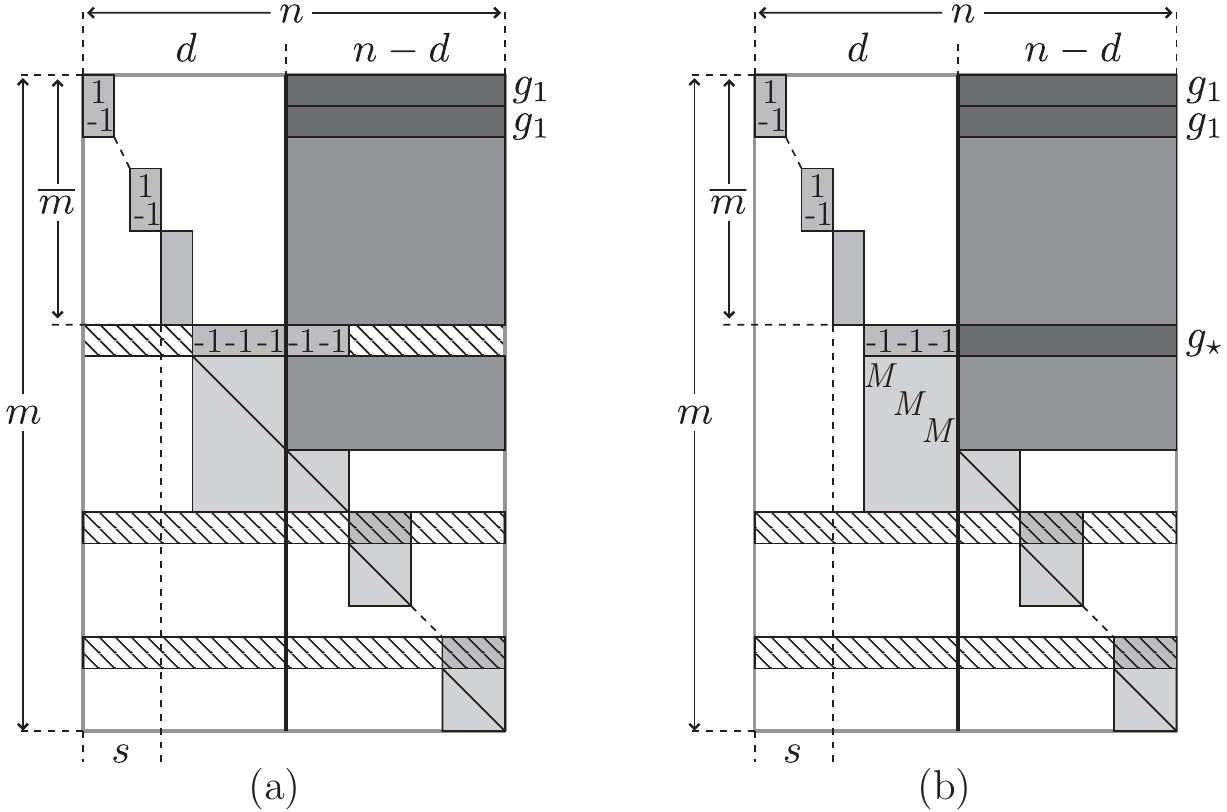}}
  \caption{Obtaining PPSN polytopes from a deformed product
    construction, when few of the factors are segments. Part (a) shows
    the technique used so far, and part (b) an additional optimization
    that exchanges a bad facet for a new vector in the Gale
    transform.}
  \label{fig:defp-ppsn2}
\end{figure}

Finally, we reformulate Proposition~\ref{prop:defp-ppsn} to express, in terms of $k$ and $\sub{n} \eqdef (n_1,\dots,n_r)$, what dimensions a $(k,\sub{n})$-PPSN polytope can have. This yields upper bounds on $\delta_{pr}(k,\sub{n})$.

\begin{theorem}\label{theo:defp-ppsn}
For any $k\ge0$ and $\sub{n} \eqdef (n_1,\dots,n_r)$ with ${1=n_1=\dots=n_s<n_{s+1}\le\dots\le n_r}$,
\[
   \delta_{pr}(k,\sub{n})
   \ \le \
   \begin{cases}
     2(k+r)-s-t   & \text{if } 3s \le 2k+2r, \\
     2(k+r-s)+1   & \text{if } 3s = 2k+2r+1, \\
     2(k+r-s+1)   & \text{if } 3s \ge 2k+2r+2,
   \end{cases}
\]
where $t\in\{s,\dots,r\}$ is maximal such that 
\[
3s+\sum_{i=s+1}^{t}(n_i+1)\ \le \ 2k+2r.
\]
\end{theorem}

\begin{proof}
Apply part (1) of Proposition~\ref{prop:defp-ppsn} when $3s\ge2k+2r+2$ and part (2) otherwise.
\end{proof}

\begin{remark}
When all the $n_i$'s are large compared to $k$, the dimension of the $(k,\sub{n})$-PPSN polytope provided by this theorem is bigger than the dimension $2k+r+1$ of the $(k,\sub{n})$-PPSN polytope obtained by the Minkowski sum of cyclic polytopes of Theorem~\ref{theo:UBminkowskiCyclic}. However, if we have many segments (neighborly cubical polytopes), or more generally if many $n_i$'s are small compared to $k$, this construction provides our best examples of PPSN polytopes.
\end{remark}

\section{Topological Obstructions}
\label{sec:topologicalObstruction}

In this section, we give lower bounds on the minimal dimension
$\delta_{pr}(k,\sub{n})$ of a $(k,\sub{n})$-PPSN polytope, applying
and extending a method developed by Sanyal~\cite{s-tovnms-09}
to bound the number of vertices of Minkowski sums of polytopes.
This method provides
lower bounds on the target dimension of any linear projection that
preserves a given set of faces of a polytope.  It uses Gale duality to
associate a certain simplicial complex $\cK$ to the set of faces that
are preserved under the projection. Then lower bounds on the
embeddability dimension of~$\cK$ transfer to lower bounds on the
target dimension of the projection. In turn, the embeddability
dimension is bounded via colorings of the Kneser graph of the system
of minimal non-faces of~$\cK$, using Sarkaria's Embeddability
Theorem. 

For the convenience of the reader, we first quickly recall this embeddability criterion.
We then provide a brief overview of Sanyal's
method before applying it to obtain lower bounds on the dimension of
$(k,\sub{n})$-PPSN polytopes.
As mentioned in the introduction, these bounds match the upper bounds obtained from our different constructions for a wide range of parameters, and thus give the exact value of the minimal dimension of a PPSN polytope.

\subsection{Sarkaria's embeddability criterion}

\subsubsection{Kneser graphs}
\label{subsubsec:kneser}

Recall that a \defn{$k$-coloring} of a graph $G=(V,E)$ is a map
$c:V\to[k]$ such that $c(u)\ne c(v)$ for $(u,v)\in E$. As usual, let
$\chi(G)$~denote the \defn{chromatic number} of $G$ (\ie the
minimal $k$ such that $G$ admits a $k$-coloring). We are interested in
the chromatic number of so-called Kneser graphs.

Let $\cZ$ be a subset of the power set $2^{[n]}$ of $[n]$. The
\defn{Kneser graph} on $\cZ$, denoted $\KG(\cZ)$, is
the graph with vertex set $\cZ$, where 
$X,Y\in\cZ$ are adjacent if and only if $X\cap Y=\emptyset$:
\[
   \KG(\cZ)
   \ \eqdef \
   \left(\cZ,\{(X,Y)\in\cZ^2\ | \  X\cap
     Y=\emptyset\}\right).
\]
Let $\KG_{n}^{k} \eqdef \KG\big({[n] \choose k}\big)$ denote the Kneser graph
on the set of subsets of $[n]$ of size~$k$.  For example, the graph
$\KG_{n}^{1}$ is the complete graph $K_n$ (of chromatic number $n$)
and the graph $\KG_{5}^{2}$ is the Petersen graph (of chromatic number
$3$).

\begin{remark}
\begin{enumerate}
\item If $n\le 2k-1$, then any two $k$-subsets of $[n]$ intersect and
  the Kneser graph $\KG_{n}^{k}$ is independent ({\it i.e.}, it has no
  edge). Thus its chromatic number is $\chi(\KG_{n}^{k})=1$.
\item If $n\ge 2k-1$, then $\chi(\KG_{n}^{k})\le n-2k+2$. Indeed, the
  map $c:{[n] \choose k}\to[n-2k+2]$ defined by
  $c(F) \eqdef \min(F\cup\{n-2k+2\})$ is a $(n-2k+2)$-coloring of
  $\KG_{n}^{k}$.
\end{enumerate}
\end{remark}

In fact, it turns out that this upper bound is the exact chromatic
number of the Kneser graph: $\chi(\KG_{n}^{k})=\max\{1,n-2k+2\}$. This
result was conjectured by Kneser~\cite{Kneser55} in 1955, and proved by
Lov\'asz~\cite{l-kccnh-78} in 1978 applying the Borsuk-Ulam
Theorem~---~see~\cite{m-ubut-03} for more details. However, we will
only need the upper bound for the topological obstruction.

\subsubsection{Sarkaria's Theorem}
\label{subsubsec:sarkaria}

Our lower bounds on the dimension of $(k,\sub{n})$-PPSN polytopes rely
on lower bounds for the dimension in which certain simplicial
complexes can be embedded. Among other possible
methods~\cite{m-ubut-03}, we use Sarkaria's Coloring and Embedding
Theorem.

We associate to any simplicial complex $\cK$ the set system
$\cZ$ of \defn{minimal non-faces} of $\cK$, that is,
the inclusion-minimal sets of
$2^{V(\cK)}\smallsetminus\cK$. For example, the
complex of minimal non-faces of the $k$-skeleton of the
$n$-dimensional simplex is ${[n+1] \choose k+2}$. Sarkaria's Theorem
provides a lower bound on the dimension into which $\cK$ can
be embedded, in terms of the chromatic number of the Kneser graph of
$\cZ$.

\begin{theorem}[Sarkaria's Theorem]
  \label{theo:sarkaria}
  Let $\cK$ be a simplicial complex embeddable in~$\R^d$,
  $\cZ$ be the system of minimal non-faces of $\cK$,
  and $\KG(\cZ)$ be the Kneser graph on $\cZ$. Then
  \[
    d
    \ \ge \
    |V(\cK)|-\chi(\KG(\cZ))-1.
  \]
\end{theorem}

In other words, we get large lower bounds on the possible embedding
dimension of $\cK$ when the Kneser graph of minimal non-faces of~$\cK$
has small chromatic number.  We refer to the excellent treatment
in~\cite{m-ubut-03} for further details.

\subsection{Sanyal's topological obstruction method}

For given integers $n>d$, we consider the orthogonal projection
$\pi:\R^n\to\R^d$ to the first $d$ coordinates, and its dual
projection $\tau:\R^n\to\R^{n-d}$ to the last $n-d$ coordinates.  Let
$P$ be a full-dimensional simple polytope in~$\R^n$, with $0$ in its
interior, and assume that its vertices are strictly preserved under
$\pi$.  Let~$F_1,\dots,F_m$ denote the facets of $P$. For all
$i\in[m]$, let $f_i$ denote the normal vector to $F_i$, and let
$g_i \eqdef \tau(f_i)$. For any face $F$ of $P$, let $\varphi(F)$ denote the
set of indices of the facets of~$P$ containing $F$, \ie such that
$F=\bigcap_{i\in \varphi(F)} F_i$.

\begin{lemma}[Sanyal \cite{s-tovnms-09}]\label{lem:gale_transform}
  The vector configuration $G \eqdef \{g_i\ | \ i\in[m]\}\subset\R^{n-d}$
  is the Gale transform of the vertex set $\{a_i\ | \ i\in[m]\}$
  of a (full-dimensional) polytope $Q$ of~$\R^{m-n+d-1}$. Up to a
  slight perturbation of the facets of~$P$, we can even assume~$Q$ to
  be simplicial.
\end{lemma}

We will refer to the polytope $Q$ as \defn{Sanyal's projection
  polytope}.  The faces of this polytope capture the key notion of
strictly preserved faces of $P$~---~remember
Definition~\ref{def:spf}. Indeed, the Projection
Lemma~\ref{lem:projection} ensures that for any face~$F$ of~$P$ that is
strictly preserved by the projection~$\pi$, the set $\{g_i\ | \ i\in
\varphi(F)\}$ is positively spanning. By Gale duality, this implies 
that the set of vertices $\{a_i\ | \ i\in[m]\smallsetminus
\varphi(F)\}$ forms a face of~$Q$.

\begin{example}
  \label{ex:entireQ}
  Let $P$ be a triangular prism in $3$-space that projects to a
  hexagon as in Figure~\ref{fig:entireQ}a, so that $n=3$, $d=2$ and
  $m=5$. The vector configuration $G\subset\R^1$ obtained by
  projecting $P$'s normal vectors consists of three vectors pointing
  up and two pointing down, so that Sanyal's projection polytope~$Q$
  is a bipyramid over a triangle. An edge $F_i\cap F_j$ of
  $P$ that is preserved under projection corresponds to the face
  $[5]\smallsetminus\{i,j\}$ of~$Q$. Notice that the six faces of $Q$
  corresponding to the six edges of~$P$ that are preserved under
  projection (in bold in Figure~\ref{fig:entireQ}a) make up the entire
  boundary complex of the bipyramid~$Q$.
\end{example}

\begin{figure}
   \centerline{\includegraphics[scale=1]{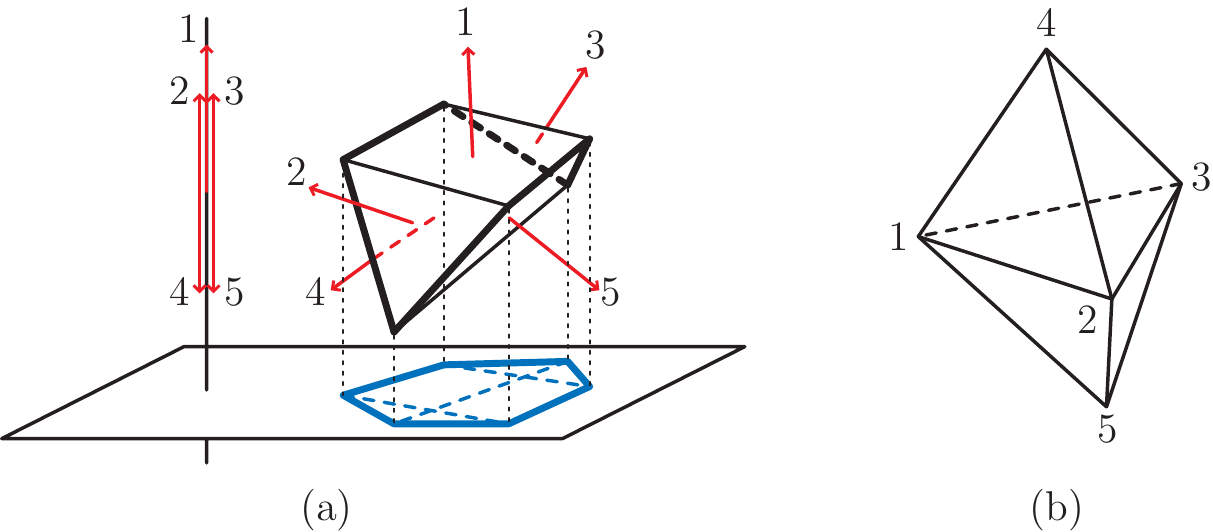}}
   \caption{(a) Projection of a triangular prism and (b) its associated
   projection polytope~$Q$.  The six faces of $Q$ corresponding to the
   six edges of~$P$ preserved under projection (bold) make
   up the entire boundary complex of~$Q$. }
   \label{fig:entireQ}
\end{figure}

Let $\cF$ be a collection of faces of $P$ that are strictly preserved
under~$\pi$.  Define $\cK$ to be the simplicial complex induced by
$\{[m]\smallsetminus \varphi(F)\ | \ F\in \cF\}$.

\begin{remark}
  Notice that not all non-empty faces of~$\cK$ correspond to non-empty
  faces in~$\cF$: in Example~\ref{ex:entireQ}, if $\cF$ consists of all strictly preserved edges, then $\cK$ is the entire
  boundary complex of Sanyal's projection polytope~$Q$, so that it contains the
  edge $\{2,3\}$. But then the complementary intersection of facets,
  $F_1\cap F_4\cap F_5$, does not correspond to any non-empty face
  of~$P$.
\end{remark}

Since the set of vertices $\{a_i\ | \ i\in[m]\smallsetminus
\varphi(F)\}$ forms a face of~$Q$ for any face $F\in\cF$, and since $Q$~is
simplicial, $\cK$~is a subcomplex of the face complex
of~$Q\subset\R^{m-n+d-1}$.  In particular, when $\cK$~is not the
entire boundary complex of~$Q$, it embeds into $\R^{m-n+d-2}$ by
stereographic projection (otherwise, it only embeds into
$\R^{m-n+d-1}$, as happens in Example~\ref{ex:entireQ}).

Thus, given the simple polytope $P\subset\R^n$ and a set $\cF$ of
faces of $P$ that we want to preserve under projection, the study of
the embeddability of the corresponding abstract simplicial complex
$\cK$ provides lower bounds on the dimension $d$ in which we can
project $P$. We proceed in the following way:
\begin{enumerate}
\item we first choose our subset $\cF$ of strictly preserved faces to
  be simple enough to understand and large enough to provide an
  obstruction;
\item we then understand the system $\cZ$ of minimal non-faces
  of the simplicial complex $\cK$;
\item finally, we find a suitable coloring of the Kneser graph on
  $\cZ$ and apply Sarkaria's 
  Theorem~\ref{theo:sarkaria} to bound the dimension in which $\cK$~can
  be embedded: a $t$-coloring of $\KG(\cZ)$ ensures that $\cK$ is not
  embeddable into $|V(\cK)|-t-2=m-t-2$, which by the previous
  paragraph bounds the
  dimension~$d$ from below as follows:
\end{enumerate}

\begin{theorem}[Sanyal \cite{s-tovnms-09}]
  \label{theo:sanyal}
  Let $P$ be a simple polytope in $\R^n$ whose facets are in general
  position, and let $\pi:\R^n\to\R^d$ be a projection. Let
  $\cF$ be a subset of the set of all strictly preserved faces of~$P$
  under~$\pi$, let $\cK$ be the simplicial complex induced by
  $\{[m]\smallsetminus\varphi(F)\ | \ F\in \cF\}$, and let $\cZ$~be
  its system of minimal non-faces. If the Kneser graph~$\KG(\cZ)$ is
  $t$-colorable, then
  \begin{enumerate}
  \item if $\cK$ is not the entire boundary complex of the Sanyal
    polytope~$Q$, then $d\ge n-t+1$;
  \item otherwise, $d\ge n-t$.\hfill$\Box$
  \end{enumerate}
\end{theorem}

In the remainder of this section, we apply Sanyal's topological
obstruction to our problem. The hope was initially to extend it to
bound the target dimension of a projection preserving the $k$-skeleton
of an arbitrary product of simple polytopes. However, the
combinatorics involved to deal with this general question turn out to
be too complicated, and so we restrict our attention to products of
simplices. This yields bounds on the minimal dimension
$\delta_{pr}(k,\sub{n})$ of a $(k,\sub{n})$-PPSN polytope.

\subsection{Preserving the $k$-skeleton of a product of simplices}

In this section, we understand the abstract simplicial complex
$\cK$ corresponding to our problem, and describe its system
of minimal non-faces.

The facets of $\simplex_{\sub{n}}$ are exactly the products
\[
   \psi_{i,j} 
   \ \eqdef \
   \simplex_{n_1}\times\dots\times\simplex_{n_{i-1}}\times
   (\simplex_{n_i}\smallsetminus\{j\})\times\simplex_{n_{i+1}}
   \times\dots\times\simplex_{n_r},  
\]
for $i\in[r]$ and $j\in[n_i+1]$.  We identify the facet $\psi_{i,j}$
with the element $j\in[n_i+1]$ of the disjoint union
$[n_1+1]\uplus[n_2+1]\uplus\dots\uplus[n_r+1]$.

Let $F \eqdef F_1\times\dots\times F_r$ be a $k$-face of
$\simplex_{\sub{n}}$. Then $F$ is contained in a facet $\psi_{i,j}$ of
$\simplex_{\sub{n}}$ if and only if $j\notin F_i$. Thus, the set of
facets of $\simplex_{\sub{n}}$ that do not contain $F$ is exactly
${F_1\uplus\dots\uplus F_r}$. Consequently, if we
want to preserve the $k$-skeleton of $\simplex_{\sub{n}}$, then the
abstract simplicial complex~$\cK$ we are interested in is induced by
\begin{equation}
  \label{eq:complexK}
   \bigg\{F_1\uplus\dots\uplus F_r
   \ \big|\  
   \emptyset\ne F_i\subset [n_i+1] \text{ for all } i\in[r],
   \textrm{ and } \sum_{i\in[r]} (|F_i|-1)=k\bigg\}.
\end{equation}
 
\begin{remark}
  \label{rem:entireQ}
  In contrast to the general case, when we want to preserve the
  \defn{complete} \mbox{$k$-skeleton} of a product of simplices, the
  complex $\cK$ cannot be the entire boundary complex of the Sanyal
  polytope~$Q$.  As a consequence, the better lower bound from part (1)
  of Sanyal's Theorem~\ref{theo:sanyal} always holds, and we 
  always use it from now on without further notice.

  To prove that $\cK$ cannot cover the entire boundary complex of~$Q$,
  observe that
  \[
    \dim Q
    \ = \
    m-n+d-1
    \ = \
    \sum (n_i+1) - \sum n_i + d - 1
    \ = \
    r+d-1,
  \]
  while $\dim\cK=r+k-1$ by~\eqref{eq:complexK}. A necessary condition
  for $\cK$ to be the entire boundary complex of $Q$ is that
  $\dim\cK=\dim Q-1$, which translates to $d=k+1$. Now suppose that
  the entire $k$-skeleton of $\simplex_{\sub{n}}$ is preserved under
  projection to dimension $k+1$. Then the projections of those
  $k$-faces are facets of $\pi(\simplex_{\sub{n}})$. Since any ridge
  of the projected polytope is contained in exactly two facets, and
  the \defn{entire} $k$-skeleton of $\simplex_{\sub{n}}$ is preserved,
  we know that any $(k-1)$-face of $\simplex_{\sub{n}}$ is also contained
  in exactly two $k$-faces.  But this can only happen if $k=n-1$,
  which means $n=d$.

  Recall from Example~\ref{ex:entireQ} that $\cK$ can be the entire
  boundary complex of~$Q$ if we do not preserve \defn{all} $k$-faces
  of~$\simplex_{\sub{n}}$.
\end{remark}

The following lemma gives a description of the minimal non-faces
of~$\cK$:

\begin{lemma}
  The system of minimal non-faces of $\cK$ is
  \[
      \cZ
      \ \eqdef \ 
      \bigg\{G_1\uplus\dots\uplus G_r\ \big|\  
        |G_i|\ne 1 \text{ for all }i\in[r], \textrm{ and } \sum_{i\,|\, 
        G_i\ne\emptyset} (|G_i|-1)=k+1\bigg\}. 
  \]
\end{lemma}

\begin{proof}
  A subset $G \eqdef G_1\uplus\dots\uplus G_r$ of
  $[n_1+1]\uplus[n_2+1]\uplus\dots\uplus[n_r+1]$ is a face of $\cK$
  when it can be extended to a subset $F_1\uplus\dots\uplus F_r$
  with $\sum (|F_i|-1)=k$ and $\emptyset\ne F_i\subset [n_i+1]$ for
  all $i\in[r]$, that is, when
  \[
    k 
    \ \ge \
    \left|\big\{i\in[r]\ |\  G_i=\emptyset\big\}\right|
    +\sum_{i\in[r]} (|G_i|-1) =
    \sum_{i\,|\,  G_i\ne\emptyset} (|G_i|-1).
  \]
  Thus, $G$ is a non-face if and only if 
  \[
    \sum_{i\, | \,  G_i\ne\emptyset} (|G_i|-1)
    \ \ge \ k+1.
  \]

  If $\sum_{i\, | \,  G_i\ne\emptyset} (|G_i|-1)> k+1$, then removing any
  element provides a smaller non-face. If there is an $i$ such that
  $|G_i|=1$, then removing the unique element of $G_i$ provides a
  smaller non-face. Thus, if $G$ is a minimal non-face, then
  $\sum_{i\, | \,  G_i\ne\emptyset} (|G_i|-1)=k+1$, and $|G_i|\ne1$ for
  all $i\in[r]$.

  Reciprocally, if $G$ is a non-minimal non-face, then it is possible
  to remove one element keeping a non-face. Let $i\in[r]$ be such that
  we can remove one element from $G_i$, keeping a non-face. Then,
  either $|G_i|=1$, or 
  \[
     \sum_{j\, | \,  G_j\ne\emptyset} (|G_j|-1)
     \ \ge \
     1+(|G_i|-2)+\sum_{j\ne i\, | \,  G_j\ne\emptyset} (|G_j|-1)\ge k+2,
  \]
  since we keep a non-face.
\end{proof}

\subsection{Colorings of $\KG(\cZ)$}

Our next goal is to provide a suitable coloring of the Kneser graph on
the system~$\cZ$ of minimal non-faces of~$\cK$.  Let $S \eqdef \{i\in[r]\ |\
n_i=1\}$ denote the set of indices corresponding to the segments, and
$R \eqdef \{i\in[r]\ |\ n_i\ge2\}$ the set of indices corresponding to the
non-segments in the product $\simplex_{\sub{n}}$.  We first provide a
coloring for two extremal situations.

\begin{theorem}[Topological obstruction for low-dimensional skeleta]
\label{theo:topObstr-small-k}
  If $k\le \sum_{i\in R} \fracfloor{n_i-2}{2}$, then the
  dimension of any $(k,\sub{n})$-PPSN polytope cannot be smaller than
  $2k+|R|+1$: 
  \[
    \delta_{pr}(k,\sub{n})
    \ \ge \
    2k+|R|+1.
  \]
\end{theorem}

\begin{proof}
Let $k_1,\dots,k_r\in\N$ be such that 
\[
   \sum_{i\in[r]} k_i
   \ = \
   k \quad\textrm{and}\quad 
   \begin{cases}
     k_i=0 & \textrm{for } i\in S;\\
     0\le k_i\le \frac{n_i-2}{2} & \textrm{for } i\in R.
   \end{cases}
\]

Observe that
\begin{enumerate}
\item such a tuple exists since $k\le \sum_{i\in R}
  \fracfloor{n_i-2}{2}$, and 
\item for any minimal non-face $G \eqdef G_1\uplus\dots\uplus G_r$ of
  $\cZ$, there exists $i\in[r]$ such that $|G_i|\ge
  k_i+2$. Indeed, if $|G_i|\le k_i+1$ for all $i\in[r]$, then
  \[
     k+1 
     \ = \
     \sum_{i\, | \,  G_i\ne \emptyset} (|G_i|-1)
     \ \le \
     \sum_{i\, | \,  G_i\ne \emptyset} k_i
     \ \le \
     \sum_{i\in[r]} k_i
     \ = \
     k,
  \]
which is impossible.
\end{enumerate}

For all $i\in[r]$, we fix a proper coloring
$\gamma_i:{[n_i+1]\choose[k_i+2]}\to[\chi_i]$ of the Kneser graph
$\KG_{n_i+1}^{k_i+2}$, with $\chi_i=1$ if $i\in S$ and $\chi_i
=n_i-2k_i-1$ if $i\in R$~---~see Section~\ref{subsubsec:kneser}. We
define a coloring $\gamma:\cZ\to[\chi_1]\uplus\dots\uplus[\chi_r]$ of
the Kneser graph on $\cZ$ as follows. Let $G \eqdef G_1\uplus\dots\uplus
G_r$ be a given minimal non-face of~$\cZ$. We arbitrarily choose an
$i\in[r]$ such that $|G_i|\ge k_i+2$, and a subset~$g$ of~$G_i$ with
$k_i+2$ elements. We color~$G$ with the color of~$g$ in
$\KG_{n_i+1}^{k_i+2}$, that is, we define $\gamma(G) \eqdef \gamma_i(g)$.

The coloring $\gamma$ is a proper coloring of the Kneser
graph~$\KG(\cZ)$. Indeed, let $G \eqdef G_1\uplus\dots\uplus G_r$
and $H \eqdef H_1\uplus\dots\uplus H_r$ be two minimal non-faces of
$\cZ$ related by an edge in $\KG(\cZ)$, which means
that they do not intersect.  Let $i\in[r]$ and $g\subset G_i$ be such
that we have colored~$G$ with~$\gamma_i(g)$, and similarly $j\in[r]$
and $h\subset G_j$ be such that we have colored~$H$
with~$\gamma_j(h)$. Since the color sets of~$\gamma_1,\dots,\gamma_r$
are disjoint, the non-faces $G$~and~$H$ can receive the same color
$\gamma_i(G)=\gamma_j(H)$ only if $i=j$ and $g$~and~$h$ are not
related by an edge in~$\KG_{n_i+1}^{k_i+2}$, which implies that $g\cap
h\ne\emptyset$. But this cannot happen, because $g\cap h\subset
G_i\cap H_i$, which is empty by assumption.  Thus, $G$ and $H$ get different
colors.

This provides a proper coloring of $\KG(\cZ)$ with $\sum
\chi_i$ colors. By Theorem~\ref{theo:sanyal} and
Remark~\ref{rem:entireQ}, we know that the dimension $d$ of the
projection is at least
\[
  \sum_{i\in[r]} n_i - \sum_{i\in[r]} \chi_i +1
  \ = \ 
  2k+|R|+1. \qedhere
\]
\end{proof}

\begin{theorem}[Topological obstruction for high-dimensional skeleta]
\label{theo:topObstr-large-k}
If $k\ge \Lfloor\frac12\sum_i n_i\Rfloor$, then any
$(k,\sub{n})$-PPSN polytope is combinatorially equivalent to
$\simplex_{\sub{n}}$:
  \[
   \delta_{pr}(k,\sub{n})
   \ \ge \ 
   \sum n_i.
  \]
\end{theorem}

\begin{proof}
  Let $G \eqdef G_1\uplus\dots\uplus G_r$ and $H \eqdef H_1\uplus\dots\uplus H_r$
  be two minimal non-faces of $\cZ$. Let $A \eqdef \{i\in[r]\ | \
  G_i\ne\emptyset \textrm{ or } H_i\ne\emptyset\}$. Then
\begin{eqnarray*}
  \sum_{i\in A} (|G_i|+|H_i|) & \ge & \sum_{G_i\ne\emptyset} (|G_i|-1)
  + \sum_{H_i\ne\emptyset} (|H_i|-1) + |A| \\ 
  & = & 2k+2+|A| 
  \ > \ 
  \sum_{i\in[r]} n_i+|A|
  \ \ge \
  \sum_{i\in A} (n_i+1).
\end{eqnarray*}
Thus, there exists $i\in A$ such that $|G_i|+|H_i|>n_i+1$, which
implies that $G_i\cap H_i\ne\emptyset$, and proves that $G\cap
H\ne\emptyset$.

Consequently, the Kneser graph $\KG(\cZ)$ is independent (and
we can color it with only one color). We obtain that the dimension $d$
of the projection is at least $\sum n_i$. In other words, in this
extremal case, there is no better $(k,\sub{n})$-PSN polytope than the
product $\simplex_{\sub{n}}$ itself.
\end{proof}

\begin{remark}\label{remark:betterColoring}
  Theorem \ref{theo:topObstr-large-k} can sometimes be strengthened a
  little: If $k = \frac{1}{2}\sum n_i-1$, and $k+1$ is not
  representable as a sum of a subset of $\{n_1,\dots,n_r\}$, then
  $\delta_{pr}(k,\sub{n})=\sum n_i$. 
\end{remark}

\begin{proof}
  As in the previous theorem, we prove that the Kneser
  graph $\KG(\cZ)$~is independent. Indeed, assume that
  $G \eqdef G_1\uplus\dots\uplus G_r$ and $H \eqdef H_1\uplus\dots\uplus H_r$ are
  two minimal non-faces of~$\cZ$ related by an edge
  in~$\KG(\cZ)$. Then, $G\cap H$ is empty, which implies that for all
  $i\in [r]$,
  \begin{equation}
    \label{eq:gihi}
    |G_i|+|H_i|
    \ \le \  n_i+1.
  \end{equation}
  Let $U \eqdef \{i\ | \ G_i\ne\emptyset\}$ and $V \eqdef \{i\ |
  \ H_i\ne\emptyset\}$. Then,
  \begin{eqnarray*}
    \sum_{i\in U\cup V} (|G_i| + |H_i|)
    &=&
    \sum_{i\in U} (|G_i|-1) + 
    \sum_{i\in V} (|H_i|-1) + |U| + |V|
    \ = \
    2k+2 + |U| + |V| 
    \\ &=&
    \sum_{i\in[r]} n_i + |U| + |V| 
    \ \stackrel{(\star)}{\ge} \
    \sum_{i\in U\cup V} n_i + |U\cup V|
    \ = \
    \sum_{i\in U\cup V} (n_i+1).
  \end{eqnarray*}
  Summing \eqref{eq:gihi} over $i\in U\cup V$ implies that both the
  inequality $(\star)$ and \eqref{eq:gihi} for $i\in U\cup V$ are in
  fact equalities. The tightness of $(\star)$ implies furthermore that
  $|U|+|V|=|U\cup V|$, so that $U\cap V=\emptyset$; in other words,
  $H_i$~is empty whenever $G_i$~is not. The equality
  in~\eqref{eq:gihi} then asserts that $|G_i|=n_i+1$ for all $i\in U$,
  and therefore
  \[
     k+1
     \ = \ 
     \sum_{i\in U} (|G_i|-1) 
     \ = \
     \sum_{i\in U} n_i
  \]
  is representable as a sum of a subset of the $n_i$, which
  contradicts the assumption. 
\end{proof}

Finally, to fill the gap in the ranges of $k$ covered by Theorems
\ref{theo:topObstr-small-k} and \ref{theo:topObstr-large-k}, 
we merge both coloring ideas as follows.

We partition $[r]=A\uplus B$ and choose $k_i\ge0$ for all $i\in A$
and $k_B\ge0$ such that
\begin{equation}
\label{eqSumKiAndKB}
\left(\sum_{i\in A} k_i \right) + k_B
\ \le \ k.
\end{equation}
We will determine the best choices for $A$, $B$, $k_B$ and the
$k_i$'s later.  Let $n_B \eqdef \sum_{i\in B}n_i$. Color the Kneser graphs
$\KG_{n_i+1}^{k_i+2}$ for $i\in A$ and $\KG_{n_B}^{k_B+1}$ with
pairwise disjoint color sets with
\[
\chi_i \eqdef 
\begin{cases}
n_i-2k_i-1 &\textnormal{if }2k_i\le n_i-2, \\
1 &\textnormal{if }2k_i\ge n_i-2, \\
\end{cases}
\]
and
\[
\chi_B \eqdef 
\begin{cases}
  0 & \textnormal{if } n_B=0, \\
  n_B-2k_B &\textnormal{if }2k_B\le n_B-1, \\
  1 &\textnormal{if }2k_B\ge n_B-1, \\
\end{cases}
\]
colors respectively. 

Observe now that for all minimal non-faces $G \eqdef G_1\uplus\dots\uplus
G_r$, either there is an $i\in A$ such that $|G_i|\ge k_i+2$, or
$\sum_{i\in B\,|\,G_i\neq\emptyset} (|G_i|-1)\ge k_B+1$.  Indeed,
otherwise
\[
   k+1 
   \ = \ \sum_{i\,|\,G_i\neq\emptyset}(|G_i|-1) 
   \ \le \
   \left(\sum_{i\in A}k_i\right) +k_B 
   \ \le \ k.
\]
This allows us to define a coloring of~$\KG(\cZ)$ in the
following way. For each minimal non-face $G \eqdef G_1\uplus\dots\uplus G_r$, we
arbitrarily choose one of the following strategies:
\begin{enumerate}
\item If we can find an $i\in A$ such that $|G_i|\ge k_i+2$, we choose
  an arbitrary subset $g$ of $G_i$ with $k_i+2$ elements, and color
  $G$ with the color of $g$ in $\KG_{n_i+1}^{k_i+2}$;

\item Otherwise, $\sum_{i\in B\,|\,G_i\neq\emptyset} (|G_i|-1)\ge k_B+1$,
  and we choose an arbitrary subset~$g$ of 
  \[
     \biguplus_{i\in
       B}(G_i\smallsetminus \{n_i+1\})
     \ \subset \ 
     \biguplus_{i\in B}[n_i]
     \]
     with $k_B+1$ elements and color $G$ with the color of $g$ in
     $\KG_{n_B}^{k_B+1}$.
\end{enumerate}
By exactly the same argument as in the proof of
Theorem~\ref{theo:topObstr-small-k}, one can verify that this provides
a valid coloring of the Kneser graph~$KG(\cZ)$ with
\[
   \chi
   \ \eqdef \
   \chi(A,B,\sub{k_i},k_B)
   \ \eqdef \ 
   \sum_{i\in A}\chi_i+\chi_B
\]
many colors.  Therefore Sanyal's Theorem \ref{theo:sanyal} and Remark
\ref{rem:entireQ} yield the following lower bound on the dimension $d$
of any $(k,\sub{n})$-PPSN polytope:
\[
  d 
  \ \ge \
  d_k
  \ \eqdef \
  d_k(A,B,\sub{k_i},k_B)
  \ \eqdef \ 
  \sum_i n_i+1-\chi
  \ \ge \ 
  \delta_{pr}(k,\sub{n}).
\]

It remains to choose parameters $A$, $B$, and $\{k_i\ | \ i\in A\}$
and $k_B$ that maximize this bound. We proceed algorithmically, by
first fixing $A$ and $B$, and choosing the $k_i$'s and $k_B$ to
maximize the bound on the dimension $d_k$. For this, we first start
with $k_i=0$ for all $i$ and $k_B=0$, and observe the variation of
$d_k$ as we increase individual $k_i$'s or $k_B$. By
(\ref{eqSumKiAndKB}), we are only allowed a total of~$k$ such
increases. During this process, we will always maintain the conditions
$2k_i\le n_i-1$ for all $i\in A_i$ and $2k_B\le n_B$ (which makes
sense by the formulas for $\chi_i$ and $\chi_B$).

We start with  $k_i=0$ for all $i$ and $k_B=0$. Then
\begin{eqnarray*}
  \chi(A,B,\sub{0},0) 
  & = &
  \sum_{i\in A}(n_i-1)+|S\cap A| + n_B 
  \\ & = &
  \sum_{i\in A} n_i-|A|+|S\cap A|+\sum_{i\in B} n_i 
  \ = \ 
  \sum_{i\in[r]} n_i-r+|B\cup S|,
\end{eqnarray*}
and
\[
  d_k(A,B,\sub{0},0) 
  \ = \
  1+r-|B\cup S|,
\]
where $S \eqdef \{i\in[r]\ | \  n_i=1\}$ denotes the set of segments. 

We now study the variation of $d_k$ as we increase each of the $k_i$'s
and $k_B$ by one.  For $i\in A$, increasing~$k_i$ by one decreases
$\chi_i$ by
\[
\begin{cases}
2,&\textnormal{ if }2k_i\le n_i-4,\\
1,&\textnormal{ if }2k_i= n_i-3,\\
0,&\textnormal{ if }2k_i\ge n_i-2,
\end{cases}
\]
and hence increases $d_k$ by the same amount. Observe in particular
that $d_k$ remains invariant if we increase $k_i$ for some segment
$i\in S$ (because $n_i=1$ for segments). Thus, it makes sense to
choose $B$ to contain all segments.
Similarly, increasing $k_B$ by one decreases $\chi_B$ by 
\[
\begin{cases}
2,&\textnormal{ if }2k_B\le n_B-3,\\
1,&\textnormal{ if }2k_B= n_B-2,\\
0,&\textnormal{ if }2k_B\ge n_B-1,
\end{cases}
\]
and  increases $d_k$ by the same amount.

Recall that we are allowed at most $k$ increases of $k_i$'s or $k_B$
by~\eqref{eqSumKiAndKB}. Heuristically, it seems reasonable to first
increase the $k_i$'s or $k_B$ that increase $d_k$ by two, and then
these that increase~$d_k$ by one. Hence we get a case distinction on
$k$, which  also depends on $A$ and $B$:

\begin{theorem}[Topological obstruction, general case]
  \label{theo:topObstr-all-k}
  Let $k\ge0$ and $\sub{n} \eqdef (n_1,\dots,n_r)$ with $r\ge1$ and
  $n_i\ge1$ for all~$i$. Let $[r]=A\uplus B$ be a partition
  of $[r]$ with $B\supset S \eqdef \{i\in[r]\ |\ n_i=1\}$.
  Define
\begin{align*}
   K_1
   & \ \eqdef \
   K_1(A,B)
   \ \eqdef \
   \sum_{i\in A}\fracfloor{n_i-2}{2} +
   \max\left\{0,\fracfloor{n_B-1}{2}\right\},  
   \\ 
   K_2
   & \ \eqdef \
   K_2(A,B)
   \ \eqdef \
   \left|\big\{i\in A\ | \ n_i\text{ is odd}\big\}\right| +
   \begin{cases}
      1 & \text{if } n_B\text{ is even and non-zero}, \\
      0 & \text{otherwise}.
   \end{cases}
\end{align*}
  Then the following lower bounds hold for the
  dimension of a $(k,\sub{n})$-PPSN polytope:
  \begin{center}
    \renewcommand{\arraystretch}{1.4}
    \begin{tabular}{ll}
    (1) If $0\le k\le K_1$, then 
    &
    $\delta_{pr}(k,\sub{n})
    \ \ge \ r+1-|B|+2k$;    
    \\
    (2) If $K_1\le k\le K_1+K_2$, then
    & 
    $\delta_{pr}(k,\sub{n})\ \ge \ r+1-|B|+K_1+k$; \\
    (3) If $K_1+K_2\le k$, then 
    &$\delta_{pr}(k,\sub{n})\ \ge \ r+1-|B|+2K_1+K_2.$
  \end{tabular}
\end{center}
\end{theorem}

This theorem enables us to recover Theorem~\ref{theo:topObstr-small-k}
and Theorem~\ref{theo:topObstr-large-k}:

\begin{corollary}
  \label{cor:topObstr}
  Let $k\ge0$ and $\sub{n} \eqdef (n_1,\dots,n_r)$ with $r\ge1$ and $n_i\ge1$
  for all~$i$, and define $S \eqdef \{i\in[r]\ |\ n_i=1\}$ and
  $R \eqdef \{i\in[r]\ |\ n_i\ge 2\}$. 
  \begin{enumerate}
  \item If 
  \[
   0\ \le \ k \ \le \ \sum_{i\in R}\fracfloor{n_i-2}{2} + 
   \max\left\{0,\fracfloor{|S|-1}{2}\right\}, 
  \]
  then $\delta_{pr}(k,\sub{n}) \ge 2k+|R|+1$.

  \item If $k\ge \Lfloor\frac12\sum n_i\Rfloor$ then $\delta_{pr}(k,\sub{n})\ge \sum_i n_i$.
  \end{enumerate}
\end{corollary}

\begin{proof}
  Take $A=R$ and $B=S$ for (1), and $A=\emptyset$ and $B=[r]$ for (2).
\end{proof}

\subsection{Explicit lower bounds} \label{sec:ExplicitLowerBounds}
There is an algorithm to explicitly choose the partitions $[r]=A\uplus
B$ which yield the best bounds in Theorem~\ref{theo:topObstr-all-k}.
Since this algorithm is quite technical, we just present the best
results we obtain with this topological obstruction. We refer to
\cite{CRM-version} for further details.

We fix $K_1=K_1(R,S)$ and define $d_0=r+1-|S|$ and $n=\sum_{i\in[r]}
n_i$. The best lower bound $d_k$ that we obtain with this coloring can
be summarized explicitly by the following case distinction~---~see
Figure~\ref{fig:GraphOfDk}:

\begin{figure}[htbp]
  \centerline{\includegraphics[scale=.9]{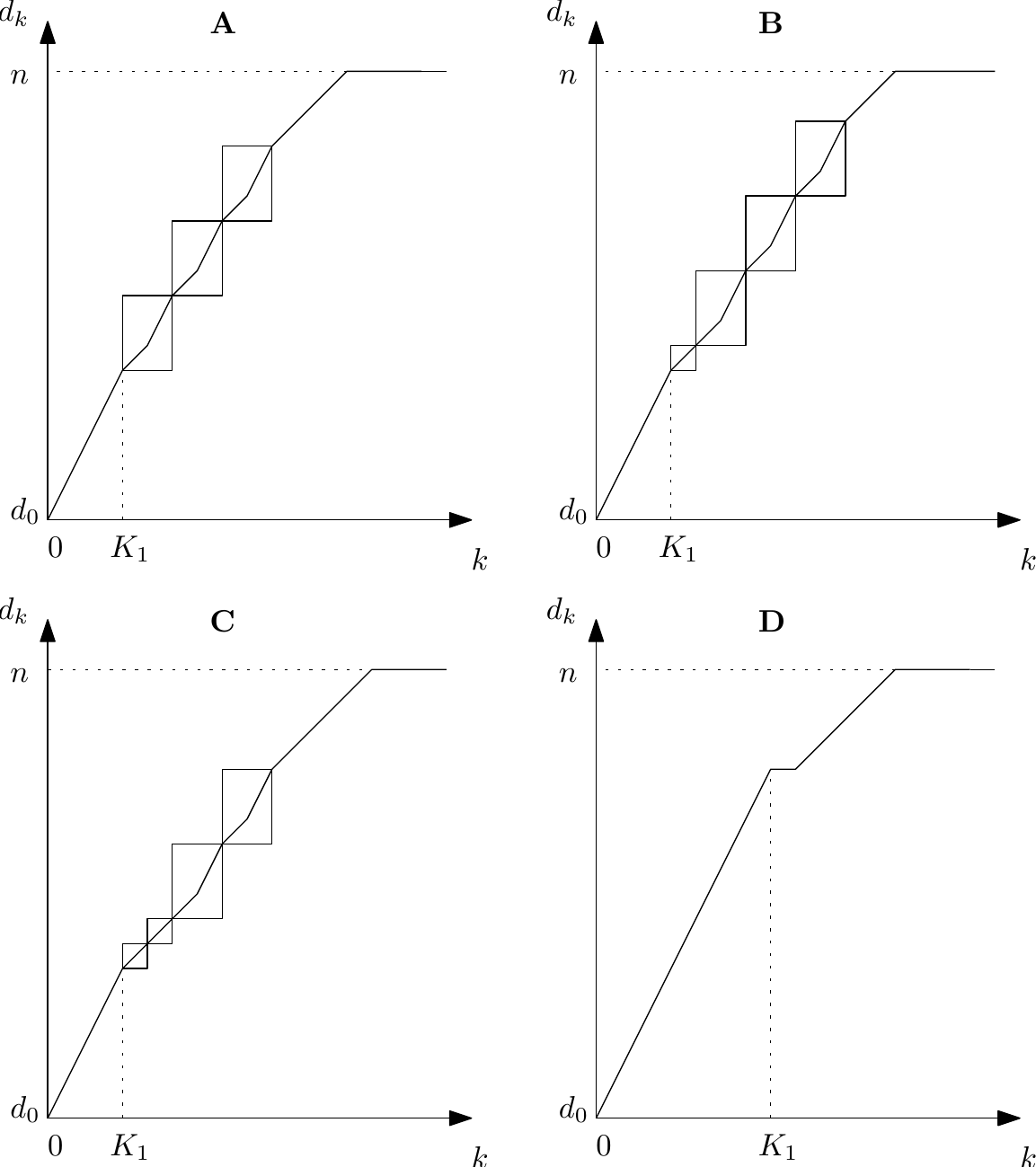}}
  \caption{\small{Four different situations for the lower bound.}}
  \label{fig:GraphOfDk}
\end{figure}

\begin{enumerate}
\item[$\mathbf{A}$.] {\it When $|S|$ is even and non-zero}: The bound
  $d_k$ increases by two for $0\le k\le K_1$. Then for each odd
  $n_i\ge 3$ we get a block with a first increment by one and a second
  increment by two. Then all increments are one until we reach the
  trivial bound $d_k=n=\sum_{i\in[r]} n_i$.
\item[$\mathbf{B}$.] {\it When $|S|$ is odd}:
As in case $\mathbf{A}$, except that the first block corresponding to
an odd $n_i\ge 3$ consists only of one increment by one.
\item[$\mathbf{C}$.] {\it When $|S|=0$ and there is an odd $n_i$}:
As in the cases $\mathbf{A}$ and $\mathbf{B}$, except that the first
two blocks corresponding to odd $n_i$'s consists only of one
increment. If there is only one odd $n_i$ then all increments from
$K_1$ on are one until we reach the trivial bound.
\item[$\mathbf{D}$.] {\it When all $n_i$ are even}: The bound $d_k$
  increases by two for $0\le k\le K_1$. The next increment is zero,
  and all further increments are one until we reach the trivial bound
  $d_k=n=\sum_{i\in[r]} n_i$.
 \end{enumerate}

\begin{remark}
Remark~\ref{remark:betterColoring} still provides a better bound for certain cases, as for example when $k=2$ and $\sub{n}=(4,2)$.
\end{remark}

\subsection{Comparison with R\"orig and Sanyal's
  results} \label{sec:ComparisonWithRSBounds} 

In \cite{rs-npps}, R\"orig and Sanyal address the special case
$n_1=\dots=n_r \eqfed n$ and $r\geq 2$. In their Theorem 4.5, they obtained
the following bound:
\[
\delta_{pr}(k,(n,\dots,n))\ \geq \ 
\begin{cases}
2k+r+1, &\textnormal{if } 0\leq k\leq r\fracfloor{n-2}{2},\\
k+\frac{1}{2}r(n-1)+1, &\textnormal{if } r\fracfloor{n-2}{2}< k\leq r\fracfloor{n-1}{2},\\
\alpha+r(n-1)+1, &\textnormal{if } r\fracfloor{n-1}{2}<k\leq rn,
\end{cases}
\]
where
$\alpha \eqdef \fracfloor{k-r\fracfloor{n-1}{2}}{\fracfloor{n+2}{2}}$. We
compare this with the graphs $C$ (if $n$ is odd) and $D$ (if $n$ is
even) of Figure~\ref{fig:GraphOfDk}. Their first case matches exactly
with the bounds of this paper, since
$K_1=r\fracfloor{n-2}{2}$. Plugging in $k=K_1$ into their first two
cases yields the same bound if $n$ is odd, but a different one if $n$
is even. If $n$ is even then the difference is $\fracfloor{r}{2}$. The
bound in their second case has slope one, that is, it increases by
one if $k$ increases by one, and the bound in their third case has
a much smaller slope. Hence the bounds of Section
\ref{sec:ExplicitLowerBounds} are stronger, especially around
$k\approx\frac{rn}{2}$. In the case $r=1$ both bounds are equal,
because at $k=K_1$ we already reach the best possible bound $rn$.

\section*{Acknowledgements}

We thank Bernardo Gonz\'alez Merino for extensive and fruitful
discussions on the material presented here.

We are indebted to Thilo R\"orig and Raman Sanyal for discussions and
comments on the subject of this paper, and their very careful reading
of an earlier draft.

We are grateful to the Centre de Re\c cerca Matem\`atica (CRM) and the
organizers of the i-Math Winter School DocCourse on Combinatorics and
Geometry, held in the Spring of 2009 in Barcelona, for having given us
the opportunity of working together during three months in a very
stimulating environment.

We would like to thank Michael Joswig for suggesting to look
at even polytopes.

Finally, we are grateful to the two anonymous referees for their
numerous remarks that helped to improve the presentation of the paper.

\bibliographystyle{alpha}
\bibliography{biblio}

\end{document}